\documentclass[10pt,a4paper]{article}

\usepackage{hyperref}
\newcommand{\footremember}[2]{%
    \footnote{#2}
    \newcounter{#1}
    \setcounter{#1}{\value{footnote}}%
}

\usepackage{a4}
\usepackage{amsmath}
\usepackage{mathptmx}       
\usepackage{helvet}         
\usepackage{courier}        
\usepackage{overpic}
\usepackage{amsfonts}
\usepackage{mathtools}
\usepackage{lipsum}
\usepackage{amsfonts}
\usepackage{amssymb}
\usepackage{graphicx}
\usepackage{epstopdf}
\usepackage{algorithmic,algorithm}
\usepackage{braket}
\usepackage{bbm}
\usepackage{stmaryrd}
\usepackage{overpic}
\usepackage{caption,subcaption}
\usepackage{xcolor}
\usepackage{amsmath}
\usepackage{nicefrac}
\usepackage{mathtools}
\usepackage{enumitem}
\usepackage[nameinlink]{cleveref}
\usepackage{grffile} 

\usepackage{varwidth}

\usepackage{stmaryrd}

\usepackage{tikz}
\usepackage{pgfplots}
\usetikzlibrary{spy,calc}

\newtheorem{remark}{Remark}

\newtheorem{assumption}{Assumption}

\newtheorem{definition}{Definition}
\newtheorem{theorem}{Theorem}
\newtheorem{proof}{Proof}
\newtheorem{lemma}{Lemma}
\newtheorem{proposition}{Proposition}
\newtheorem{corollary}{Corollary}

\crefname{hypothesis}{Hypothesis}{Hypotheses}
\crefname{assumption}{Assumption}{Assumption}
\makeindex          
\newcommand{\R}{\mathbb{R}}
\newcommand{\N}{\mathbb{N}}
\newcommand{\EE}{\mathbb{E}}
\newcommand{\E}{\mathcal{E}}
\newcommand{\YY}{\mathcal{Y}}
\newcommand{\hH}{H_{\textup{av}}^1(D)}
\newcommand{\PP}{\mathbb{P}}
\newcommand{\Go}{\partial D}%{\partial X}
\newcommand{\tr}{\text{tr}}

\renewcommand{\d}{\textup{d}}
\newcommand{\n}{\textup{n}}
\newcommand{\dx}{\textup{d}x}
\newcommand{\dr}{\textup{d}r}
\newcommand{\ds}{\textup{d}s}

\newcommand{\dt}{\textup{d}t}

\DeclareMathOperator*{\esssup}{ess\,sup}
\newcommand{\id}{\operatorname{id}}

\newcommand{\kin}{{\kappa_{\text{in}}}}
\newcommand{\kout}{{\kappa_{\text{out}}}}
\newcommand{\Din}{{D_{\text{in}}}}
\newcommand{\Dout}{{D_{\text{out}}}}
\newcommand{\divv}{{\textup{div}}}
\renewcommand{\L}{\textup{L}}
\newcommand{\UU}{{\mathcal{U}}}
\newcommand{\kappaHeart}{{\kappa_{\textup{heart}}}}
\newcommand{\kappaLungs}{{\kappa_{\textup{lungs}}}}
\newcommand{\kappaTrunk}{{\kappa_{\textup{trunk}}}}
\newcommand{\kappaHeartdet}{{\bar{\kappa}_{\textup{heart}}}}
\newcommand{\kappaLungsdet}{{\bar{\kappa}_{\textup{lungs}}}}
\newcommand{\kappaTrunkdet}{{\bar{\kappa}_{\textup{trunk}}}}

\begin{document}

\title{Stochastic approximation for optimization in shape spaces}
  
\author{Caroline Geiersbach\footremember{a}{University of Vienna, \texttt{caroline.geiersbach@univie.ac.at}} \and Estefania Loayza-Romero\footremember{b}{Chemnitz University of Technology, \texttt{estefania.loayza@math.tu-chemnitz.de}} \and Kathrin Welker\footremember{c}{Helmut-Schmidt-University / University of the Federal Armed Forces Hamburg, \texttt{welker@hsu-hh.de}}
}

\date{}

\maketitle

\abstract{
In this work, we present a novel approach for solving stochastic shape optimization problems. Our method is the extension of the classical stochastic gradient method to infinite-dimensional shape manifolds. We prove convergence of the method on Riemannian manifolds and then make the connection to shape spaces. The method is demonstrated on a model shape optimization problem from interface identification. Uncertainty arises in the form of a random partial differential equation, where underlying probability distributions of the random coefficients and inputs are assumed to be known. We verify some conditions for convergence for the model problem and demonstrate the method numerically.
}

\section{Introduction}
Shape optimization involves the identification of a shape with optimal response properties. This subject has enjoyed active research for decades due to its many applications, particularly in engineering; see for instance \cite{Pironneau1984, Sokolowski1991} for an introduction. A challenge in shape optimization is in the modeling of shapes, which do not inherently have a vector space structure. Various models of the space of shapes and associated metrics have been used in the literature. 
Recently,  \cite{Schulz}	made a link between shape calculus and shape manifolds, and thus enabled the usage of optimization techniques on manifolds in the context of shape optimization. One possible approach is to cast the sets of shapes in a Riemannian viewpoint, where each shape is a point on an abstract manifold equipped with a notion of distances between shapes 
(cf., e.g.,~\cite{MichorMumford2,MichorMumford1,srivastava2016functional,younes2010shapes}).
In \cite{MichorMumford}, a survey of various suitable inner products is given, e.g., the curvature weighted metric and the Sobolev metric. 
From a theoretical and computational point of view, it is attractive to optimize in Riemannian shape manifolds because algorithmic ideas from \cite{Absil} can be combined with approaches from differential geometry.
In contrast to \cite{Absil}, in which only optimization on finite dimensional manifolds is discussed, \cite{Ring2012} considers also infinite-dimensional manifolds.
In this setting, the shape derivative can be used to solve such shape optimization problems using the gradient descent method. In the past, e.g., \cite{Delfour-Zolesio-2001,Sokolowski1991}, major effort in shape calculus has been devoted towards expressions for shape derivatives in the Hadamard form, i.e., in the boundary integral form. An equivalent and intermediate result in the process of deriving Hadamard expressions is a volume expression of the shape derivative, called the weak formulation. One usually has to require additional regularity assumptions in order to transform volume into surface forms.
In addition to saving analytical effort, this makes volume expressions preferable to Hadamard forms, which is utilized in e.g.~\cite{Langer-2015}. One possible approach to use these formulations is given in \cite{SchulzSiebenbornWelker2015:2}; an inner product called the Steklov--Poincar\'{e} metric is proposed, which we also use in this work.

Until recently,  the models used in the area of shape optimization have been deterministic, i.e., all physical quantities were supposed to be known exactly. However, many relevant problems involve a constraint in the form of a partial differential equation (PDE), which contains inputs or material properties that may be unknown or subject to uncertainty.  PDEs under uncertainty have been well-investigated in the literature; see for instance \cite{Lord2014} for an introduction and \cite{Martinez-Frutos2018} for their application to optimization, including their use in problems in shape optimization with the level set method.  Increasingly, stochastic models are being used in shape optimization with the goal of obtaining more robust solutions.  A number of works has focused on structural optimization with either random Lam\'{e} parameters or forcing \cite{allaire2015deterministic,Atwal2012,Conti2008, Conti2018, Dambrine2015, Martinez-Frutos2016}. Stochastic models have also handled uncertainty in the geometry of the domain \cite{Bruegger2018, Harbrecht2018, Liu2017}. To ensure well-posedness of the stochastic problem, either an order must be defined on the relevant random variables, as in \cite{Conti2018}, or the problem needs to be transformed to a deterministic one by means of a probability measure. One possibility is to compute the worst case design \cite{Bellido2017,Dambrine2016}. Another possibility is to use first and second order moments to cast the problem in a deterministic setting \cite{Dambrine2015}; this is particularly relevant if the probability distribution of the underlying random variable is unknown. The sum of expectation and standard deviation is sometimes used \cite{Martinez-Frutos2016}, but this fails to be a coherent risk measure. The most popular choice in the literature is the (risk-neutral) measure expectation. This measure, which we also consider in this work, is appropriate when the cost associated with the shape's failure is of little concern. For other choices for probability measures, a review can be found in \cite{Rockafellar2015}.
 
The development of efficient algorithms for shape optimization under uncertainty is an active area of research. If the number of possible scenarios in the underlying probability space is small, then the optimization problem can be solved over the entire set of scenarios. This approach is not relevant for most applications, as it becomes intractable if the random variable has more than a few scenarios. Algorithmic approaches for shape optimization problems under uncertainty involve the use of a standard deterministic solver in combination with either a discretization of the stochastic space or using an ensemble/sample from the stochastic space. The former approach includes the stochastic Galerkin method, used on random domains in \cite{Eigel2019} and polynomial chaos, applied to topology optimization in  \cite{Keshavarzzadeh2017}. Ensemble-based approaches involve taking independent realizations or carefully chosen quadrature points of the random variable. The most basic method is sample average approximation (SAA), also known as the Monte Carlo method, where a random sample is generated once and the the original problem is replaced by the sample average problem over the fixed sample.

Recently, stochastic approximation (SA) methods have been proposed to efficiently solve PDE-constrained optimization problems involving uncertainty \cite{Geiersbach2019, Geiersbach2020, Martin2018, Haber2012}. This approach is fundamentally different from the methods already mentioned, since sampling is performed dynamically as part of the optimization procedure. Because of its use of partial function information in the form a so-called stochastic gradient, it has a low computational cost when compared to other methods. In this paper, we present a novel use of the stochastic gradient method, namely for PDE-constrained shape optimization problems under uncertainty. In \cref{sec:Algorithm}, we prove convergence of the method on a Riemannian manifold based on the work on finite-dimensional manifolds by \cite{Bonnabel2011} and infinite-dimensional Hilbert spaces by \cite{Geiersbach2020}. Additionally, we make the connection to optimization on shape spaces. In \cref{sec:model}, we develop a model problem, which is motivated by applications to electrical impedance tomography. Moreover, we verify shape differentiability for the model problem as well as bounds on the second moment of the stochastic gradient, which are necessary for the convergence of the algorithm presented in \cref{sec:Algorithm}. We show a numerical simulation in \cref{sec:Numerics}. Closing remarks are presented in \cref{sec:conclusion}.

\section{Stochastic approximation in shape spaces}
\label{sec:Algorithm}
The principal aim of this section is the presentation of stochastic approximation to iteratively solve a shape optimization problem containing uncertain parameters and inputs in a suitable shape space.  The section is organized as follows. First, in \cref{subsection_infinite-dimensional-MF} we highlight some of the difficulties in working with infinite-dimensional shape spaces. Our main result is in \cref{subsection:SGAlgorithmManifolds}, where we prove convergence of the stochastic gradient method on a Riemannian manifold. Then, we introduce a manifold of shapes with an appropriate metric (cf.~\cref{subsection_ShapeSpaces}). In \cref{subsection_ShapeCalcStochModeling}, we give new results for shape calculus combined with stochastic modeling. 

\subsection{Infinite-dimensional shape manifolds}
\label{subsection_infinite-dimensional-MF}
The shape space we use in this paper is the space of plane unparametrized curves, i.e., the space of all smooth embeddings of the unit circle in the plane modulo reparametrizations, which is an infinite-dimensional manifold and is usually denoted as $B_e$\footnote{This space is defined later in (\ref{B_e}).}. Our choice of this space comes from the fact that in shape optimization, the set of permissible shapes generally does not allow a vector space structure. This is a central difficulty in the formulation of efficient optimization methods for these applications. In particular, without a vector space structure, there is no obvious distance measure. If one cannot work in vector spaces, shape spaces that allow a Riemannian structure are the next best option. However, they come with additional difficulties; in our choice of shape space, we are working with an infinite-dimensional manifold.
As mentioned in~\cite{bauer2014overview}, while working in this kind of space, many difficulties arise and there are still open questions. For one, most of the Riemannian metrics defined over these spaces are weak and hence the gradient is not necessarily defined.
Furthermore, the existence and uniqueness of solutions of the geodesic equation are not guaranteed and need to be checked for each metric; this means in some cases the exponential map is not well-defined. In some pathological cases, it is also possible that the exponential map fails to be a diffeomorphism on any neighborhood, see, e.g.,~\cite{constantin2007geodesic}. Finally, any assumption regarding the injectivity radius is challenging to prove in practice, and to the authors' knowledge has not been studied for the space of plane curves. 

Another problem involves the fact that distances on an infinite-dimensional Riemannian manifold can be degenerate. 
In \cite{MichorMumford1}, it is shown that the reparametrization invariant $L^2$-metric on the infinite-dimensional manifold of smooth planar curves induces a geodesic distance equal to zero. 
In \cite{MichorMumford1}, a curvature weighted $L^2$-metric is employed as a remedy and it is proven that the vanishing phenomenon does not occur for this metric. Several Riemannian metrics on this shape space are examined in further publications, e.g., \cite{BauerHarmsMichor,MichorMumford2,MichorMumford}. All these metrics arise from the $L^2$-metric by putting weights, derivatives or both in it. In this manner, we get three groups of metrics: almost local metrics (cf.~\cite{Bauer,BauerHarmsMichor_SobolevII,MichorMumford}), Sobolev metrics (cf.~\cite{BauerHarmsMichor,MichorMumford}) and weighted Sobolev metrics (cf.~\cite{BauerHarmsMichor_SobolevII}).
It can be shown that all these metrics do not induce the phenomenon of vanishing geodesic distance under special assumptions, which are given in the publications mentioned.  Summarizing, working with infinite-dimensional manifolds is very challenging and remains an active area of research.

\subsection{Stochastic gradient method on manifolds}
\label{subsection:SGAlgorithmManifolds}
In the following, we introduce notation from differential geometry and probability theory; for detailed definitions of the introduced objects, we refer to the literature \cite{Lang,Lee,Gut2013}.

Let $(\mathcal{U}, G)$ be a (possibly infinite-dimensional) connected manifold equipped with a Riemannian metric, i.e., a smoothly varying family of inner products $G =(G_u)_{u \in \mathcal{U}}$. Let $\lVert \cdot \rVert:= \sqrt{G(\cdot,\cdot)}$ denote the induced norm. 
The triple $(\Omega, \mathcal{F}, \PP)$ denotes a probability space, where $\mathcal{F} \subset 2^{\Omega}$ is the $\sigma$-algebra of events and $\PP\colon \Omega \rightarrow [0,1]$ is a probability measure. A random vector $\xi: \Omega \rightarrow \Xi \subset \R^m$ is given; sometimes we use the notation $\xi \in \Xi$ to denote a realization of the random vector. We are focused on problems of the form
$$\min_{u \in \mathcal{U}} \left\lbrace j(u):= \EE[J(u,\xi)] = \int_\Omega J(u,\xi(\omega)) \,\d\PP(\omega)\right\rbrace,$$
where $J\colon \mathcal{U} \times \Xi \rightarrow \R$ is a functional such that $\omega \mapsto J(u,\xi(\omega))$ is well-defined and $\PP$-integrable for all $u \in \mathcal{U}$, i.e., the expectation above is finite on the manifold. 
We denote the tangent space at a point $u \in \mathcal{U}$ by $T_u \mathcal{U}$, defined in its geometric version as $T_u\UU=\{ c:\R \rightarrow \UU: c \text{ differentiable}, c(0) =u\}/\overset{u}{\sim}$, where $c_1 \overset{u}{\sim} c_2$ means $c_1$ is $u$-equivalent\footnote{If $\{(U_\alpha, \phi_\alpha)\}_\alpha$ is the atlas of $\UU$, two differentiable curves $c_1, c_2:\R \rightarrow \UU$ with $c_1(0) = c_2(0) =u$ are called $u$-equivalent if $\tfrac{\d}{\dt}\phi_{\alpha}(c_1(t))\vert_{t=0}  =\tfrac{\d}{\dt} \phi_{\alpha}(c_2(t))\vert_{t=0}$ holds for all $\alpha$ with $u \in U_\alpha$.} 
 to $c_2$. The derivative of a scalar field $j\colon\mathcal{U} \rightarrow \R$ at $u$ in the direction $v  \in T_u \mathcal{U}$ is defined by the pushforward. For each point $u\in \mathcal{U}$, the pushforward associated with $j$ is given by the map
$$(j_\ast)_u\colon T_u\mathcal{U}\to \R,$$
with
$$(j_\ast)_uv:= \frac{\d}{\d t} j(c(t)) \vert_{t=0} = (j \circ c)'(0)$$
for $v=\dot{c}(t)\in T_u\mathcal{U}$, where $c\colon I\to \mathcal{U}$ is a differentiable curve and $I\subset\R$ is an interval.

A Riemannian gradient $\nabla j(u) \in T_u \mathcal{U}$ is defined by the relation
$$
(j_\ast)_uw
 = G_u(\nabla j(u), w) \quad \forall w \in T_u \mathcal{U}.$$
Of course, in an infinite-dimensional setting, there is no guarantee that the gradient is well-defined, as already mentioned in \cref{subsection_infinite-dimensional-MF}.
The Hessian of $j$ at $u$ is defined by $\text{Hess}j(u)[v]:=\nabla_v^{\text{cov}} j(u)$, where  $\nabla_v^{\text{cov}}$ denotes the covariant derivative in the direction $v$. We now define the stochastic gradient.
\begin{definition}\label{definition:stochastic-gradient}
Let $J:\mathcal{U} \times \Xi \rightarrow \R$ be a functional defined on the manifold $(\mathcal{U}, G)$ and $j:\mathcal{U} \rightarrow \R$ be given by $j(u) = \EE[J(u,\xi)]$. For a fixed realization $\xi \in \Xi$, set $J_\xi(\cdot):=J(\cdot,\xi)$. The \emph{stochastic gradient} of $j$ in a point $u \in \mathcal{U}$ is a $\PP$-integrable function $\nabla J:\mathcal{U} \times \Xi \rightarrow T \mathcal{U}$ such that\\
1) For almost every $\xi \in \Xi$, $
((J_\xi)_\ast)_uw
 = G_u (\nabla J(u,\xi),w)$ for all  $w\in T_u \mathcal{U},$ \\
2) $\EE[\nabla J(u,\xi)] = \nabla j(u).$
\end{definition}
In a slight abuse of notation, we will always use $\nabla J(u,\xi)$ to denote the gradient with respect to the $u$ variable. 

In order to locally reduce an optimization problem on a manifold to an optimization problem on its tangent space, we need the concept of the exponential map, and its approximation, the so-called retraction. We denote the exponential mapping at $u$ by $\exp_{u}\colon T_{u}\mathcal{U}\to \mathcal{U},\,v\mapsto \exp_{u}(v)$, which assigns to every tangent vector $v$ the value $\gamma(1)$ of the geodesic $\gamma\colon [0,1]\to \mathcal{U}$ satisfying $\gamma(0) = u$ and $\dot{\gamma}(0) = v.$
A retraction is denoted by $\mathcal{R}_u\colon T_u\mathcal{U}\to \mathcal{U}$ satisfying $\mathcal{R}_u(0_u)=u$ and the so-called local rigidity condition $d\mathcal{R}_u(0_u)=\text{id}_{T_u\mathcal{U}}$, where $0_u$ denotes the zero element of $T_u \UU$. 

We now formulate a stochastic gradient method on manifolds. 
This method dates back to a paper by Robbins and Monro \cite{Robbins1951}, where an iterative method for finding the root of a function was introduced, which used only estimates of the function values. The main advantages of this method include its low memory requirements, low computational complexity, as well as ease of implementation along deterministic gradient-based solvers. Stochastic gradient methods have been widely used in applications and its study on manifolds remains an active area of research~\cite{Bonnabel2011,ZhangSra2016}. 

The algorithm is shown in \cref{alg:stochastic_descent-manifold}. We will work with the standard step-size rule
\begin{equation}
 \label{eq:Robbins-Monro-step-sizes}
 t_n \geq 0, \quad \sum_{n=1}^\infty t_n = \infty, \quad \sum_{n=1}^\infty t_n^2 < \infty.
\end{equation}
 
\begin{remark}
This ``Robbins--Monro'' step-size rule was originally introduced in the paper \cite{Robbins1951}. The rule provides the appropriate scaling for the stochastic gradient in order to ensure sufficient decrease (on average) in the objective function, while asymptotically dampening variance. To guarantee convergence to a stationary point for a stochastic gradient method with no other variance reduction technique, this rule is crucial. Other choices of step-sizes, such as those obtained using an Armijo backtracking procedure, generally fail in stochastic approximation, as demonstrated in \cite[Example 1.1]{Geiersbach2020a}.
\end{remark}

\begin{algorithm}
	\begin{algorithmic}[1]
		\STATE \textbf{Initialization:} Choose $u_1 \in \mathcal{U}$
		\FOR{$n=1,2,\dots$}
		\STATE Generate $\xi_n \in \Xi$, independent of $\xi_1, \dots, \xi_{n-1}$
		%\STATE Choose $t_n$ satisfying \cref{eq:Robbins-Monro-step-sizes}
		\STATE Set $u_{n+1} := \exp_{u_n}(-t_n \nabla J(u_n,\xi_n))$
		\ENDFOR
	\end{algorithmic}
	\caption{Stochastic gradient method on manifolds}
	\label{alg:stochastic_descent-manifold}
\end{algorithm}

Now, we analyze the convergence of \cref{alg:stochastic_descent-manifold}.
Let the length of a curve $c$ be denoted by $\L(c) = \int_0^1 \|c'(t)\|\dt$. Then the distance $\d:\mathcal{U}\times\mathcal{U}\rightarrow \R$ between points $u, \tilde{u}$ on the manifold is given by
	\[
	\d(u,\tilde{u}) = \inf \{\L(c) \colon c  \text{ is a piecewise smooth curve on } \mathcal{U} \text{ from } u \text{ to } \tilde{u}\}.
	\]
	We denote the parallel transport along the geodesic $\gamma\colon [0,1]\to \mathcal{U}$ by $P_{\alpha,\beta}\colon T_{\gamma(\alpha)}\mathcal{U}\to T_{\gamma(\beta)}\mathcal{U}$. 
	The injectivity radius $i_u$ at a point $u \in \mathcal{U}$ is defined as $i_u := \sup\{r > 0 : \exp_u \vert_{B_r(0_u)} \text{ is a diffeomorphism}\}$, where $B_r(0_u) \subset T_u \mathcal{U}$ is a ball with radius $r$. The injectivity radius $i(\mathcal{U})$ of the manifold $\mathcal{U}$ is defined by
	$$i(\mathcal{U}) := \inf_{u \in \mathcal{U}} i_u.$$ 
	
To show convergence, we make the following fundamental assumptions about the manifold $(\mathcal{U},G)$. 
	\begin{assumption}
		\label{assump:manifold}
		We assume that
		\begin{enumerate}
			\item \label{assump:gradient} For all $u \in \mathcal{U}$ and almost all $\xi \in \Xi$, the gradient $\nabla J(u,\xi)$ exists;
			\item \label{assump:non_deg_dist} The distance $\d(\cdot,\cdot)$ is non-degenerate;
\item \label{assump:injectivity-radius} The manifold $(\mathcal{U},G)$ has a positive injectivity radius $i(\mathcal{U})$;
\item \label{assump:exp_geo} For all $u\in\mathcal{U}$ and $\tilde{u} \in B_{i_u}(0_u)$, the minimizing geodesic between $u$ and $\tilde{u}$ is completely contained in $B_{i_u}(0_u)$. 
		\end{enumerate}
	\end{assumption}
Thanks to \cref{assump:manifold}, the geodesic between $u$ and $\tilde{u}$ such that $\d(u,\tilde{u}) \leq i(\mathcal{U})$ is uniquely defined  and there exists a $v \in T_u \mathcal{U}$ such that $\tilde{u} = \exp_u(v)$.

\begin{remark}
As mentioned in \cref{subsection_infinite-dimensional-MF}, these assumptions, while very natural for finite-dimensional manifolds, are not automatically satisfied for \discretionary{infinite-dimen-}{sional}{infinite-dimensional} manifolds. 
At first glance, \cref{assump:exp_geo} of~\cref{assump:manifold} appears to be superfluous. However, this is only true for finite dimensional manifolds. There are examples as stated in~\cite[p.~19]{escherright}, in which the minimizing geodesics can leave the neighborhood where the exponential map is a diffeomorphism.  We remark that the requirement that the manifold has a positive injectivity radius is quite strong for infinite-dimensional manifolds. 
\end{remark}

\begin{definition}
 \label{def:Lipschitz-gradient}
Let $(\mathcal{U}, G)$  be a connected Riemannian manifold with a positive injectivity radius. We call a function $j:\mathcal{U} \rightarrow \R$ \emph{$L$-Lipschitz continuously differentiable} if its gradient exists for all $u\in\mathcal{U}$, and there exists a $L>0$ such that for all $u,\tilde{u} \in \mathcal{U}$ with $\d(u,\tilde{u}) \leq i(\mathcal{U})$,
 \begin{equation}\label{eq:Lipschitz-condition}
\lVert P_{1,0}\nabla j(\tilde{u})-\nabla j(u)\rVert \leq L \d(u,\tilde{u}),
\end{equation}
where $P_{1,0}:T_{\gamma(1)} \mathcal{U} \rightarrow T_{\gamma(0)} \mathcal{U}$  is the parallel transport along the unique geodesic such that $\gamma(0)=u$ and $\gamma(1) = \tilde{u}.$
\end{definition}

Now we are in the position to present the first result, which is needed for the convergence proof. 

\begin{theorem}\label{theorem:Lipschitzcondition}
Let  $(\mathcal{U}, G)$ satisfy \cref{assump:manifold} and let $u,\tilde{u} \in \mathcal{U}$ be such that $\d(u,\tilde{u}) \leq i(\mathcal{U})$. If $j$ is $L$-Lipschitz continuously differentiable, then with $v = \exp_u^{-1}(\tilde{u})$ it follows that 
\begin{equation}
\label{eq:LS-condition}
j(\tilde{u}) -j(u)\leq G_u(\nabla j(u),v) + \frac{L}{2}\|v\|^2.
	\end{equation}
\end{theorem}

\begin{proof}
We consider the mapping $\phi\colon [0,1]\to\ \R,\,t\mapsto j(\exp_u(tv))$. The derivative of a geodesic curve $\gamma$ is given by $\gamma'(t)=P_{0,t}\gamma'(0)$ (cf. \cite[p.~310]{FerSva}). By the chain rule, we get 
\begin{equation}
\label{phi_derivative}
\phi'(t)=(j_\ast)_{\exp_{u}(t v)}P_{0,t}v =  G_{\gamma(t)}(\nabla j(\exp_u(tv)),P_{0,t}v).
\end{equation}
Since the parallel transport is an isometry, we have
$$G_{{\gamma(t)}}(\nabla j( \exp_u(tv)),P_{0,t}v) = G_{u}(P_{t,0} \nabla j( \exp_u(tv)), P_{t,0} P_{0,t}v),$$ and additionally $P_{0,t}^{-1}=P_{t,0}$ (cf.~\cite[p.~308]{FerSva}), so \cref{phi_derivative} gives
\begin{equation}
\label{phi_derivative_isom}
\phi'(t)=G_{u}(P_{t,0}\nabla j(\exp_u(tv)),v).
\end{equation}

 Thanks to~\cref{assump:manifold},~ \cref{assump:injectivity-radius}, the exponential mapping has a well-defined inverse $\exp_u^{-1}(\tilde{u}):\mathcal{U} \rightarrow T_u \mathcal{U}$ such that $\d(\tilde{u},u) = \lVert \exp_u^{-1}(\tilde{u})\rVert$. Thus $\d(\exp_u(tv), u) = t\lVert v\rVert$. 
 Now we rewrite the fundamental theorem of calculus in the form
\begin{equation}
\label{fund_the_calculus}
\phi(1)-\phi(0)=\phi'(0)+\int_{0}^{1}\phi'(t)-\phi'(0)\,\dt
\end{equation}
with the aim of invoking \eqref{eq:Lipschitz-condition}. With $\phi(0)= j(u)$, $\phi(1)= j(\exp_u(v))$, and $\phi'(0)= G_{u}(\nabla j(u),v)$, we get by \eqref{fund_the_calculus} that
\begin{align*}
j(\exp_u(v))-j(u) & = G_u(\nabla j(u),v) +\int_{0}^{1}  G_u(P_{t,0}\nabla j( \exp_u(tv)) - \nabla j(u),v) \, \dt\\
&\leq G_u(\nabla j(u),v)+ \int_{0}^{1}  \lVert P_{t,0}\nabla j( \exp_u(tv)) - \nabla j(u)\rVert \lVert v\rVert \,\dt\\
&\hspace{-.1cm}\stackrel{\cref{eq:Lipschitz-condition}}{\leq} G_u(\nabla j(u),v) + \int_{0}^{1} L t\|v \|^2  \,\dt\\
&= G_u(\nabla j(u),v) + \frac{L}{2}\|v\|^2.
\end{align*}
\end{proof}

For the convergence proof, we recall that a sequence $\{ \mathcal{F}_n\}$ of increasing sub-$\sigma$-algebras of $\mathcal{F}$ is called a filtration. A stochastic process $\{\beta_n\}$ is said to be adapted to the filtration if $\beta_n$ is $\mathcal{F}_n$-measurable for all $n$. If  $\mathcal{F}_n = \sigma(\beta_1, \dots, \beta_n),$\footnote{The $\sigma$-algebra generated by a random variable $\beta:\Omega \rightarrow \R$ is given by $\sigma(\beta) = \{ \beta^{-1}(B): B \in \mathcal{B}\}$, where $\mathcal{B}$ is the Borel $\sigma$-algebra on $\R$. Analogously, the $\sigma$-algebra generated by the set of random variables $\{ \beta_1, \dots, \beta_n\}$ is the smallest $\sigma$-algebra such that $\beta_i$ is measurable for all $i=1, \dots, n.$}
we call $\{ \mathcal{F}_n\}$ the natural filtration. Furthermore, we define for a $\PP$-integrable random variable $\beta:\Omega \rightarrow \R$ the conditional expectation $\EE[\beta | \mathcal{F}_n]$, which is a random variable that is $\mathcal{F}_n$-measurable and satisfies $\int_A \EE[\beta(\omega) | \mathcal{F}_n] \, \d \PP(\omega) = \int_A \beta(\omega) \,\d \PP(\omega)$ for all $A \in \mathcal{F}_n$. Sometimes we use the notation $\EE_{\xi}[\cdot]$ to emphasize that the expectation is computed with respect to $\xi.$  If an event $F \in \mathcal{F}$ is satisfied with probability one, i.e.~$\PP[F] = 1$, we say $F$ occurs almost surely and denote this with a.s. We will use the following results, the proofs of which can be found in \cite{Pflug1996}, Appendix L and \cite{Metivier2011}, Theorem 9.4, respectively. We use the notation $\beta^- := \max \{ 0 -\beta\}.$

\begin{lemma}[Robbins--Siegmund] \label{lemma:Robbins-Siegmund}
Let $\{\mathcal{F}_n\}$ be an increasing sequence of $\sigma$-algebras and 
$v_n$, $a_n$, $b_n$, $c_n$ nonnegative random variables adapted to $\mathcal{F}_n$ for all $n$. If
\begin{equation}\label{eq:Robbins-Siegmund}
\EE[v_{n+1} | \mathcal{F}_n] \leq v_n(1+a_n)+ b_n-c_n,
\end{equation}
and $\sum_{n=1}^\infty a_n < \infty,  \sum_{n=1}^\infty b_n < \infty$ a.s., then with probability one, $\{v_n\}$ is convergent and $\sum_{n=1}^\infty c_n < \infty$. 
\end{lemma}

\begin{lemma}[Quasimartingale convergence theorem]
\label{lemma:Quasimartingale}
Let $\{\mathcal{F}_n\}$ be an increasing sequence of $\sigma$-algebras and $v_n$ be a real-valued random variable adapted to $\mathcal{F}_n$ for all $n$ satisfying the following conditions:
\begin{enumerate}
 \item $\sum_{n=1}^\infty \EE[ |\EE[v_{n+1}| \mathcal{F}_n] - v_n| ] < \infty$ and
 \item $\sup_n \EE[v_n^-] < \infty.$
\end{enumerate}
Then the sequence $\{v_n \}$ converges a.s.~to a $\PP$-integrable random variable $v_\infty$ and
$\EE[v_\infty] \leq \liminf_n \EE[|v_n|]<\infty.$
\end{lemma}

We are now ready for our main convergence result. We base our analysis on the contribution by \cite{Bonnabel2011}, who proved convergence of the stochastic gradient method for finite-dimensional manifolds, and by \cite{Geiersbach2020a}, who proved convergence for infinite-dimensional Hilbert spaces. 

\begin{theorem}
\label{thm:convergence}
Let $(\mathcal{U}, G)$  satisfy \cref{assump:manifold}. Suppose that the sequence $\{ u_n\}$ generated by \cref{alg:stochastic_descent-manifold} is $\mathcal{F}_n$-measurable and a.s.~contained in a bounded set $\mathcal{C}$. On an open set $U \subset \mathcal{U}$ containing $\mathcal{C}$, $j:U \rightarrow \R$ is assumed to be $L$-Lipschitz continuously differentiable and bounded below. Suppose that $\nabla J(u,\xi)$ is a stochastic gradient according to \cref{definition:stochastic-gradient} and there exists a nonnegative constant $M$ such that $\EE[\lVert \nabla J(u,\xi)\rVert^2] \leq M$ for all $u$. 
\begin{enumerate}
 \item Then, the sequence $\{j(u_n)\}$ converges a.s. and $\liminf_{n \rightarrow \infty} \lVert \nabla j(u_n) \rVert = 0.$
 \item If additionally, $f(u):= \lVert \nabla j(u) \rVert^2$ is $L_f$-Lipschitz continuously differentiable, then  $\lim_{n \rightarrow \infty} \nabla j(u_n) = 0$ a.s. In particular, (strong) limit points of $\{ u_n\}$ are stationary points of $j$.
\end{enumerate}
\end{theorem}

\begin{remark}
We relax several assumptions from \cite{Bonnabel2011}; in particular, we do not require that the objective function is three times continuously differentiable, requiring twice continuous differentiability and a Lipschitz condition on the second order derivative. Most importantly, we do not require the stochastic gradient to be uniformly bounded, which precludes many choices of random variables, but impose instead a bound on the variance. Finally, as our application involves an infinite-dimensional manifold, we relax the assumption of compactness. We note that $u_n$ is automatically $\mathcal{F}_n$-measurable if $\{ \mathcal{F}_n\}$ is the natural filtration induced by the sequence $\{ \xi_n\}$ from \cref{alg:stochastic_descent-manifold}. The requirement that $\{ u_n\}$ stays in a bounded set $\mathcal{C}$ is not automatic: this can be enforced by the use of regularizers or follows for certain choices of $j$; see \cite{Bottou1998, Davis2018}. 	 As a note, the properties given in \cref{assump:manifold} only need to apply to the subset $U$ for which $j$ is $L$-Lipschitz continuously differentiable. 
\end{remark}

\begin{proof}[Proof of \cref{thm:convergence}]
Without loss of generality assume $j \geq 0$ (otherwise with $\bar{j}:=\inf_{u \in \mathcal{C}} j(u)$ observe $\tilde{j}:=j - \bar{j}$ and make the same arguments for $\tilde{j}$). First, we argue that there is an index $N$ such that $\d(u_{n+1},u_n) \leq i(\mathcal{U})$ for all $n\geq N$. Notice that by Jensen's inequality and the assumption on the stochastic gradient, it holds for all $u \in \mathcal{C}$ that
\begin{equation}
\label{eq:Jensens-proof-convergence}
\lVert \EE [\nabla J(u,\xi)] \rVert \leq \EE[\lVert \nabla J(u,\xi)\rVert] \leq \sqrt{M}.
\end{equation}
It follows that the stochastic gradient is bounded in expectation and hence bounded with probability one. Therefore there exists an index $N$ such that $\d(u_{n+1},u_n) \leq i(\mathcal{U})$ for all $n\geq N$.
Let $v_n:=\nabla J(u_n,\xi_n).$ Using the update given by \cref{alg:stochastic_descent-manifold}, and the fact that~\eqref{eq:Lipschitz-condition} is satisfied, \cref{theorem:Lipschitzcondition} implies that 
\begin{equation}
 \label{eq:Lipschitz-inequality}
\begin{aligned}
j(u_{n+1})-j(u_n) &\leq -t_n G(\nabla j(u_n), v_n) + \tfrac{1}{2} Lt_n^2  \lVert v_n\rVert^2.
\end{aligned}
\end{equation} 
Taking conditional expectation on both sides of \cref{eq:Lipschitz-inequality}, we get by monotonicity of the conditional expectation and measurability of $u_n$ with respect to $\mathcal{F}_n$ that
\begin{equation}
 \label{eq:Lipschitz-inequality-conditional-expectation}
\begin{aligned}
\EE[j(u_{n+1})|\mathcal{F}_n]-j(u_n) &\leq -t_n \EE[G(\nabla j(u_n), v_n)] | \mathcal{F}_n] + \tfrac{1}{2} L t_n^2 \, \EE[\lVert v_n\rVert^2 | \mathcal{F}_n].
\end{aligned}
\end{equation}
Since $\xi_n$ is chosen independently of $\xi_1, \dots, \xi_{n-1}$ by \cref{alg:stochastic_descent-manifold}, it follows that $\EE[v_n|\mathcal{F}_n] = \EE_{\xi}[\nabla J(u_n,\xi)] = \nabla j(u_n)$ for all $n$. 
The expression \cref{eq:Lipschitz-inequality-conditional-expectation} simplifies to
\begin{equation}
 \label{eq:Lipschitz-inequality-expectation}
 \EE[j(u_{n+1})|\mathcal{F}_n] \leq j(u_n)   -t_n \lVert \nabla j(u_n) \rVert^2 + \tfrac{1}{2}L M t_n^2.
\end{equation}
Now, with $a_n = 0$, $b_n = \tfrac{1}{2} LM t_n^2$, and $c_n = t_n\lVert \nabla j(u_n) \rVert^2$, we get by \cref{lemma:Robbins-Siegmund} that the sequence $\{ j(u_n)\}$ is a.s.~convergent and additionally that
$\sum_{n=1}^\infty t_n \lVert \nabla j(u_n) \rVert^2 < \infty$
with probability one. In particular, it follows that $\liminf_{n \rightarrow \infty}  \lVert \nabla j(u_n)\rVert^2 = 0$ a.s. This proves the first statement.

For the second part, we first show that 
\begin{equation}
\label{eq:expectation-steps-finite-sum}
\sum_{n=1}^\infty t_n \EE[\lVert \nabla j(u_n) \rVert^2] < \infty.
\end{equation}
Taking expectation on both sides of \eqref{eq:Lipschitz-inequality-expectation}, summing, and rearranging, we get 
\begin{equation}
 \label{eq:finite-expectation-sum}
 \begin{aligned}
 \sum_{n=1}^{\tilde{N}} t_n \EE[\lVert \nabla j(u_n) \rVert^2] &\leq  \sum_{n=1}^{\tilde{N}} \left( \EE[j(u_n)] - \EE[j(u_{n+1})] +  \frac{L M t_n^2}{2}\right)\\
 &\leq \EE[j(u_1)]- \bar{j} +   \sum_{n=1}^{\tilde{N}}  \frac{L M t_n^2}{2}.\\
 \end{aligned}
\end{equation}
Notice that the right-hand side of \eqref{eq:finite-expectation-sum} is bounded as $\tilde{N}\rightarrow \infty$ due to the step-size condition \eqref{eq:Robbins-Monro-step-sizes} and the left-hand side is monotonicity increasing in $\tilde{N}$. Therefore, by the monotone convergence theorem, we obtain \eqref{eq:expectation-steps-finite-sum}. Now, we note that, by similar arguments to those used in \cite[Appendix B]{Bonnabel2011},
\begin{align*}
G_u(v,\nabla f(u)) = G_u(\nabla j(u),2\text{Hess}j(u)[v])
\end{align*}
and since the Hessian operator is self-adjoint~\cite[Lemma 11.1]{Lee2018}, we get that $\nabla f(u) = 2\text{Hess}j(u)[\nabla j(u)].$ Since $\nabla f$ is $L_f$-Lipschitz continuous, it follows by \cref{theorem:Lipschitzcondition} that
\begin{equation}
\label{eq:Lipschitz-inequality-second}
\begin{aligned}
f(u_{n+1}) - f(u_n) &\leq -2 t_n G(v_n, \text{Hess}j(u_n) [\nabla j(u_n)]) + \tfrac{1}{2}L_f t_n^2  \lVert v_n \rVert^2.
\end{aligned}
\end{equation}
Taking conditional expectation on both sides of \eqref{eq:Lipschitz-inequality-second}, we get
\begin{equation}
\label{eq:final-inequality-proof}
\begin{aligned}
\EE[f(u_{n+1})|\mathcal{F}_n] - f(u_n) &\leq -2 t_n G(\nabla j(u_n),\text{Hess}j(u_n)[ \nabla j(u_n)]) + \tfrac{1}{2}  L_f M t_n^2\\
&\leq 2 t_n L \lVert \nabla j(u_n) \rVert^2 + \tfrac{1}{2} t_n^2 L_f M,
\end{aligned}
\end{equation}
where in the last step, we used $\lVert \text{Hess} j(u_n) \rVert \leq L$ by $L$-Lipschitz continuity of $j$. Taking the expectation on both sides of \eqref{eq:final-inequality-proof}, we have
$$\EE[\EE[f(u_{n+1})|\mathcal{F}_n] - f(u_n)] \leq 2 t_n \EE[\lVert \nabla j(u_n) \rVert^2] L    +\tfrac{1}{2} t_n^2 L_f M.$$

Now, we can verify the conditions of \cref{lemma:Quasimartingale} with $v_n = f(u_n)$. Obviously, we have $\sup_n \EE[f(u_n)^-] < \infty.$ The terms on the right-hand side of \cref{eq:final-inequality-proof} are summable by the first part of the proof and \eqref{eq:expectation-steps-finite-sum}. Therefore, by \cref{lemma:Quasimartingale} we get that $f(u_n) = \lVert \nabla j(u_n) \rVert^2$ converges almost surely. Since we already established $\liminf_{n \rightarrow \infty} \lVert \nabla j(u_n)\rVert^2= 0$, we obtain $\lim_{n \rightarrow \infty}\lVert \nabla j(u_n)\rVert^2=0.$ This implies that with probability one, $\lim_{n \rightarrow \infty} \nabla j(u_n) = 0.$ 
\end{proof}

The following proposition can be proven using the same arguments as in \cite{Bonnabel2011}.

\begin{proposition}
 \label{proposition:convergence-retractions}
With the same assumptions as in \cref{thm:convergence}, let $\mathcal{R}_u$ be a twice differentiable retraction and replace line 5 of \cref{alg:stochastic_descent-manifold} by the update 
\begin{equation}
\label{eq:retraction-step}
u_{n+1}= \mathcal{R}_{u_n}(-t_n \nabla J(u_n,\xi_n)). 
\end{equation}
Then, with probability one, $\{j(u_n)\}$ converges~and $\lim_{n\rightarrow \infty} \nabla j(u_n) = 0$.
\end{proposition}

\subsection{The shape space $B_e$}
\label{subsection_ShapeSpaces}
In this paper, we focus on the manifold of one-dimensional smooth shapes, which we introduce next. Of course, to apply the results from \cref{subsection:SGAlgorithmManifolds}, one can also choose other shape spaces with a Riemannian structure.
First, we introduce notation. Let $D\subset \mathbb{R}^2$ be a bounded Lipschitz domain with boundary $\partial D$. The domain $D$ is assumed to be partitioned into two subdomains $\Din$ and $\Dout$ in such a way that $ \Din \subset D$ and $\Dout\subset D$ and $\Din \sqcup u \sqcup \Dout = D$, where $\sqcup$ denotes the disjoint union. The interior boundary $u\coloneqq\partial \Din $ is assumed to be smooth and the outer boundary is denoted by $\partial D$.  We use standard notation for Sobolev spaces $H^r(D)$ with corresponding norms $\lVert \cdot \rVert_{H^r(D)}.$ The notation $H_0^r(D)$ indicates the subspace of $H^r(D)$ containing functions equal to zero on the boundary. Additionally, $H_0^r(D,\R^2)$ denotes a vector-valued Sobolev space and its seminorm and norm are denoted by $\lvert \cdot\rvert_{H^r(D,\R^2)}$ and $\lVert \cdot \rVert_{H^r(D, \R^2)}$, respectively. The space of $k$-times continuously differentiable functions $f:D \rightarrow \R^2$ a.e.~vanishing on the boundary is denoted by $C_0^k(D,\R^2)$. The inner product between two vectors $v,w \in \R^2$ is denoted by $v\cdot w = v_1 w_1 + v_2 w_2$. The Euclidean norm is denoted by $\|\cdot\|_2$ and $\id$ denotes the ($2\times 2$) identity matrix.

We concentrate on one-dimensional shapes in this paper. The \emph{space of one-dimen\-sional smooth shapes} (cf.~\cite{MichorMumford1}) is characterized by the set
\begin{equation}
\label{B_e}
B_e:=B_e(S^1,\mathbb{R}^2):=\text{Emb}(S^1,\mathbb{R}^2) / \text{Diff}(S^1),
\end{equation}
i.e., the orbit space of $\mathrm{Emb}(S^1,\mathbb{R}^2)$ under the action by composition from the right by the Lie group $\mathrm{Diff}(S^1)$. 
Here, $\mathrm{Emb}(S^1,\R^2)$ denotes the set of all embeddings from the unit circle $S^1$ into $\R^2$, which contains all simple closed smooth curves in $\R^2$. Note that we can think of smooth shapes as the images of simple closed smooth curves in the plane of the unit circle because the boundary of a shape already characterizes the shape. The set $\mathrm{Diff}(S^1)$ is the set of all diffeomorphisms from $S^1$ into itself, which characterize all smooth reparametrizations. These equivalence classes are considered because we are only interested in the shape itself and images are not changed by reparametrizations. 
In \cite{KrieglMichor}, it is proven that the shape space $B_e$ is a smooth manifold; together with appropriate inner products it is even a Riemannian manifold.
In order to define a suitable metric, we need the tangent spaces of $B_e.$
The tangent space $T_u B_e$ is isomorphic to the set of all smooth normal vector fields along $u\in B_e$, i.e.,
\begin{equation}
\label{isomorphismTcBe}
T_u B_e\cong\left\{h\colon h=\alpha \n,\, \alpha\in \mathcal{C}^\infty(u)\right\} \cong \left\{\alpha: \alpha \in C^\infty(u) \right\},
\end{equation}
where the symbol $\n$ denotes the exterior unit normal field to the shape $u$. Following the ideas presented in~\cite{SchulzSiebenbornWelker2015:2}, we choose the Steklov--Poincar\'e metric defined below.
\begin{definition}
Let $\textup{tr}\colon  H^1_0(D,\mathbb{R}^2) \to H^{1/2}(u,\mathbb{R}^2)$ denote the trace operator on Sobolev spaces for vector-valued functions and $a_u\colon H_0^1(D,\R^2) \times H_0^1(D,\R^2) \rightarrow \R$ be a symmetric and coercive bilinear form. If $V\in H^1_0(D,\mathbb{R}^2)$ solves the Neumann problem
\begin{equation}\label{weak-elasticity-N2}
a_u(V,W)=\int_{u} v (\textup{tr}(W))\cdot \n\ \ds\quad \forall\hspace{.3mm}  W\in H^1_0(D,\mathbb{R}^2),
\end{equation}
and $S^{pr}\colon  H^{-1/2}(u) \to H^{1/2}(u),\
v \mapsto (\textup{tr}(V))\cdot \n$ denotes the projected Poincar\'e--Steklov operator, then the Steklov--Poincar\'e metric is defined by the mapping
$$G^S\colon H^{1/2}(u)\times H^{1/2}(u)  \to \mathbb{R},
(v,w) \mapsto 
\int_{u} v (S^{pr})^{-1}w \ \ds.$$ 
\end{definition}
To define a metric on $B_e$, we restrict the Steklov--Poincar\'e metric to the mapping $G^S\colon  T_u B_e \times T_u B_e \rightarrow \R$. In the next section, we will relate the manifold $B_e$ to the shape derivative to obtain shape gradients to be used in \cref{alg:stochastic_descent-manifold}. It is worth mentioning that some of the following considerations are only of a formal nature.
In view of developing a numerical procedure, we are working with vector fields which are less smooth. To be more precise, if $v \in B_e$, then $v$ should be smooth. However, the Steklov--Poincar\'e metric definition deals only with $H^{1/2}$-functions.
Of course, it would be possible to consider other types of metrics like the Soblolev-type metrics or almost local metrics. However, in order to obtain an efficient shape optimization algorithm, the Steklov-Poincar\'{e} metric has some numerical advantages over these other metrics as shown in 
\cite{SchulzSiebenborn2016,Siebenborn2017,Welker2016}.

\subsection{Shape calculus combined with stochastic modeling}
\label{subsection_ShapeCalcStochModeling}
In this section, we generalize the shape derivative for expectation functionals and give conditions under which the shape derivative and expectation can be exchanged. Additionally, we make the connection between shape calculus and the shape space presented in \cref{subsection_ShapeSpaces}.

There are different approaches for the representation of perturbed shapes. 
The perturbation of identity is defined for a given vector field $V$ and $T>0$ as a family of mappings $\{F_t^V\}_{t\in[0,T]}$ such that $F_t^V\colon \bar{D}\to\mathbb{R}^2$, $F_t^V(x):= x+tV(x)$ for all $x\in \bar{D}.$ 
\footnote{It is possible to guarantee invertibility of the perturbation of identity. More precisely, if $V$ is Lipschitz continuous in $\bar{D}$ and for a $c>0$, $\lVert \nabla V \rVert_{C(D,\R^{2,2})} < c$, then there exists a $T>0$ such that for all $t\in [0,T]$, $F_t^V$ is invertible, see, e.g.,~\cite[footnote~9, p.~25]{Sturm2014}.} For a given subset $A$ of $D$, we define
\begin{equation}
\label{eq:perturbation_of_identity}
F_t^V(A):=\{ F_t^V(x) \colon x \in A\}.
\end{equation}
Alternatively, the perturbations could be described as the flow $F_t(x)\coloneqq \zeta(t,x)$ determined by the initial value problem
\begin{equation*}
\frac{\partial\zeta}{\partial t}(t,x)=V(\zeta(t,x)), \quad \zeta(0,x)=x
\end{equation*}
i.e., by the velocity method. In this work, we focus on the perturbation of identity. 
Now we can introduce the definition of the shape derivative for a fixed realization.

\begin{definition}[Shape derivative for a fixed realization]
\label{definition_ShapeDer}
Let $D \subset \R^2$ be open and the realization $\xi \in \Xi$ be fixed. Moreover, let $k \in \bar{\N}$ and $u \subset D$ be (Lebesgue) measurable. The Eulerian derivative of a shape functional $J(\cdot,\xi)$ at $u$ in the direction $V\in C_0^k(D,\R^2)$ is defined (if it exists) by 
	\begin{equation}\label{eq:Eulerianderivative}
	d J(u,\xi)[V]:= \lim_{t \rightarrow 0^+} \frac{J(F_t^V(u),\xi) - J(u, \xi)}{t}.
	\end{equation}
	If for all directions $V\in C_0^\infty(D,\R^2)$, the Eulerian derivative \eqref{eq:Eulerianderivative} exists and the mapping $V \mapsto d J(u,\xi)[V]\colon C_0^\infty(D,\R^2) \to \mathbb{R}$
	is linear and continuous, then $J(\cdot,\xi)$ is called shape differentiable.
\end{definition}

We will show under what conditions $j(\cdot) = \EE[J(\cdot,\xi)]$ is shape differentiable in $u$. 
\begin{lemma}
\label{lemma:differentiability-expectation}
Suppose that $J(\cdot,\xi)$ is shape differentiable in $u$ for almost every $\xi \in \Xi$. Assume there exists a $\tau > 0$ and a $\PP$-integrable real function $C:\Xi \rightarrow \R$ such that for all $t \in [0,\tau]$, all $V \in C_0^\infty(D,\R^2)$, and almost every $\xi$,
\begin{equation}
\label{eq:difference-quotient-random-shape-derivative}
R_t^V(\xi):=\frac{J(F_t^V(u),\xi) - J(u,\xi)}{t} \leq C(\xi).
\end{equation}
Then $j$ is shape differentiable in $u$ and 
\begin{equation}
\label{eq:differentiability-expectation-exchange}
dj(u)[V] = \EE[d J(u,\xi)[V]] \quad \forall V\in C_0^\infty(D,\R^2).
\end{equation}
\end{lemma}

\begin{proof}
Since $J(u,\xi)$ is shape differentiable, the limit $\lim_{t \rightarrow 0^+} R_t^V(\xi) = dJ(u,\xi)[V]$ exists for all $V$. We have that $|R_t^V(\xi)|\leq C(\xi)t$ and $C(\xi)$ is integrable, i.e.~$\EE[C(\xi)]< \infty$. By Lebesgue's dominated convergence theorem, we thus get
\begin{align*}
\quad \lim_{t \rightarrow 0^+} \int_\Omega R_t^V(\xi(\omega)) \, \d \PP(\omega) &= \int_{\Omega} \lim_{t \rightarrow 0^+} R_t^V(\xi(\omega)) \,\d \PP(\omega)\\
 \Leftrightarrow \quad dj(u)[V]= \lim_{t\rightarrow 0^+} \frac{j(F_t^V(u))-j(u)}{t} &= \int_\Omega dJ(u,\xi(\omega))[V]\, \d \PP(\omega) = \EE[dJ(u,\xi)[V]].
\end{align*}
Therefore \eqref{eq:differentiability-expectation-exchange} holds. Linearity and continuity of $V \mapsto dj(u)[V]$ follows by linearity and continuity of $V \mapsto dJ(u,\xi)[V]$ for almost every $\xi$.
\end{proof}

\begin{remark}
The arguments used in the proof of \cref{lemma:differentiability-expectation} can be applied to vector fields of lower regularity to obtain conditions for exchanging the Eulerian derivative and expectation.
\end{remark}

Now, we will make the connection between shape calculus and shape spaces. 
From now on, we will denote the shape space $\mathcal{U}:=B_e$ with corresponding metric $G:=G^S$, i.e.~$(\mathcal{U},G) = (B_e,G^S)$. We define the set $\UU_D:=\{ u\in \UU, u \subset D\}$ of shapes $u$ belonging to the manifold $\UU$ that are also contained in the hold-all domain $D$. We will allow $u$ to vary, so one should keep in mind that $D$ depends on $u$, i.e., $D=D(u)$. If $u$ is changing, then the subdomain $\Din\subset D$ changes in a natural manner.

As utilized in~\cite{SchulzSiebenborn2016,Schulz2015a,SchulzSiebenbornWelker2015:2,Welker2016}, the Steklov--Poincar\'{e} metric allows the computation of the Riemannian shape gradient as a representative of the shape derivative in volume form. Besides saving analytical effort during the calculation process of the shape derivative, this technique is computationally more efficient than using an approach which needs the surface shape derivative form (cf., e.g., \cite{Siebenborn2017,Welker2016}). The shape derivative defined in~\cref{definition_ShapeDer} can be given in the boundary (strong) and the volume (weak) representation. 
The Hadamard structure theorem \cite[Theorem~2.7]{Sokolowski1991} states the existence of a scalar distribution $r$ on the shape $u$. 
We assume $r\in L^2(u)$. Thus, the shape derivative in its strong form can be expressed by $dJ(u,\xi)[W]=\int_{u} r W\cdot \n \, \ds$.  In this setting, a representation $\tilde{v}\in T_{u} \mathcal{U}$ of the Riemannian shape gradient in terms of the inner product $G$ on the manifold is the solution to 
$$G(\tilde{v},w)=\left(r,w\right)_{L^2(u)} \, \quad\forall w\in T_{u} \mathcal{U}.$$
From this, we get that the vector $V\in H_0^1(D,\R^2)$ can be viewed as an extension of a Riemannian shape gradient to the hold-all domain $D$ because of the identities 
\begin{equation}
\label{Steklov_identity}
G(v,w)= dJ(u,\xi)[W]=a(V,W)\quad \forall W\in H_0^1(D,\R^2),
\end{equation}
where $v=\text{tr}(V) \cdot \n,w=\text{tr}(W) \cdot \n$. In general, $v,w$ are not necessarily elements of $T_u {\mathcal{U}}$ because it is not ensured that $V,W\in H_0^1(D,\R^2)$ are $\mathcal{C}^\infty$.
As mentioned above, these elements need to be considered only formally.

In \cref{Steklov_identity}, one option for $a(\cdot, \cdot)$ is the bilinear form associated with linear elasticity, i.e.,
\begin{equation}
\label{eq:linear-elasticity-bilinear-form}
a^{\text{elas}}(V, W):=\int_D (\lambda \text{tr}  ( \epsilon(V)  )\text{id} + 2  \mu \epsilon(V)) : \epsilon(W)\, \dx,
\end{equation}
where $\epsilon(W):= \frac{1}{2} \, (\nabla W + \nabla W^T)$, $A : B$ denotes the Frobenius inner product for two matrices $A, B$ and $\lambda,\mu \in \R$ denote the Lam\'{e} parameters.

\begin{remark}
\label{remark:convergence-manifold-implies-convergence-holdall}
It is straightforward to show that $a^{\text{elas}}(\cdot,\cdot)$ is a bounded and coercive bilinear form.
By \cref{Steklov_identity}, and by coercivity, there exists a $k>0$ and by boundedness, there exists a $K>0$ such that
$$  k \lVert V\rVert_{H^1(D,\R^2)}^2 \leq a^{\text{elas}}(V,V) = G(v,v) = \lVert v \rVert^2 \leq K \lVert V\rVert_{H^1(D,\R^2)}^2.$$
\end{remark}

To summarize, we extend the stochastic gradient $\nabla J(u,\xi)$, defined on the tangent space of the manifold (from line 5 of \cref{alg:stochastic_descent-manifold}), to the hold-all domain by solving the following \emph{deformation equation}: find $V = V(u,\xi) \in H_0^1(D,\R^2)$ s.t.
\begin{equation}
a^{\text{elas}}(V, W) = d J(u,\xi)[W] \quad \forall W\in H_0^1(D,\R^2).
\label{deformatio_equation}
\end{equation}
The negative solution $-V$ is a descent direction for $J(u,\xi)$ since 
$$dJ(u,\xi)[-V] = a^{\text{elas}}(V,-V) = -\lVert V\rVert_{H^1(D,\R^2)} \leq 0.$$

\section{Application to an interface identification problem}
\label{sec:model}
In this section, we formulate the stochastic shape optimization model, which we use to demonstrate \cref{alg:stochastic_descent-manifold}. The problem under consideration is an interface identification problem and has been studied in a number of texts  \cite{Buttazzo2003,Ito-Kunisch-Peichl,Sokolowski1991}. A motivation for this model is in electrical impedance tomography, where the material distribution of electrical properties such as electric conductivity and permittivity inside the body is to be determined \cite{Cheney1999,Kwon2002}. Moreover, electrical impedance tomography  is also considered in case of uncertain boundary conditions in \cite{dambrine2019incorporating}.

\subsection{Model formulation}
\label{subsection:model-formulation}
In the model, we allow for randomness in the material properties and random boundary inputs. For each random source, it is assumed that the probability distribution is known, for example by priorly obtained empirical samples.

We allow for uncertainty in material constants and boundary conditions by definition of a probability space $(\Omega, \mathcal{F}, \PP)$. The probability space is to be understood as a product space $(\Omega, \mathcal{F}, \PP) = (\Omega_g \times \Omega_\kappa, \mathcal{F}_g \times \mathcal{F}_\kappa,  \PP_{g} \times \PP_{\kappa})$. We define a boundary input function $g\colon\Go \times \Omega_g \rightarrow \R$ and a material coefficient
\begin{equation}\label{eq:randomfield}
\kappa\colon D \times \Omega_\kappa \rightarrow \R, (x,\omega) \mapsto \kin(\omega)\mathbbm{1}_{\Din}(x)+\kout(\omega)\mathbbm{1}_{\Dout}(x),
\end{equation}
 where $\kappa_i:\Omega \rightarrow \Xi_\kappa^i \subset \R$ are independent random variables and $\mathbbm{1}_{D_i}$ denotes the indicator function of the set $D_i$, for $i\in\{\text{in}, \text{out}\}$. To facilitate simulation, we make a standard finite-dimensional noise assumption. This is automatically satisfied for $\kappa$ with $\xi_\kappa(\omega) := (\kin(\omega),\kout(\omega))$. For $g$, we assume there exists a $m$-dimensional vector  $\xi_g(\omega) := (\xi_g^1(\omega), \dots, \xi_g^m(\omega))$ of real-valued, independent random variables $\xi_g^i:\Omega \rightarrow \Xi_g^i \subset \R$ such that
 $$g(x,\omega) = g(x,\xi_g(\omega)) \quad  \text{on } D \times \Omega.$$
To simplify notation, we set $\xi:=(\xi_\kappa, \xi_g)$, $\Xi := \Xi_\kappa^{\text{in}} \times \Xi_\kappa^{\text{out}} \times \Xi_g^1 \times \cdots \times \Xi_g^m$ and now write $\kappa(\xi) = \kappa(\cdot,\xi)$ and $g(\xi) = g(\cdot,\xi)$ for a given $\xi \in \Xi$.

Let $\bar{y}\colon D\to\R$ denote (deterministic) measurements and $\nu>0$ be a given constant. The outward normal vector to $D$ and the outward normal vector to $D_{\text{in}}$ are both denoted by $\n$. We define the objective functional for a fixed realization $\xi \in \Xi$ by
\begin{equation}
\label{objective}
J(u,\xi):=J^\text{obj}(u, \xi)+\nu J^\text{reg}(u),
\end{equation}
where 
\begin{align}
\label{Objective_TrackingType}
J^\text{obj}(u, \xi)  :=\frac{1}{2} \int_D (y(x,\xi)  - \bar{y}(x))^2 \, \dx\quad\text{ and }\quad J^\text{reg}(u) :=  \int_{u}  \ds.
\end{align}
The model problem subject to a random PDE in the strong form is as follows:
\begin{align}
\min_{u \in \mathcal{U}} \quad  \EE \Big[ J^\text{obj}(u, \xi) \Big] &+ \nu J^\text{reg}(u)\label{eq:problem} \\
\text{s.t.} \,\,\, y\colon D\times \Xi\to \R,&\, (x,\xi)\mapsto y(x,\xi) \text{ satisfies} \qquad \nonumber\\
- \nabla \cdot (\kappa \nabla y) &=  0,\quad \text{in } D \times \Xi \label{eq:PDE1} \\
\kappa \frac{\partial y}{\partial \n} &= g, \quad \text{in } \Go \times \Xi. \label{eq:PDE2}
\end{align}
The following continuity conditions are imposed for the state and flux at the interface:
\begin{equation}\label{eq:jumpconditions}
\left\llbracket \kappa \frac{\partial y}{\partial \n} \right\rrbracket = 0, \quad \llbracket y \rrbracket  = 0, \quad \text{in } u \times \Xi.
\end{equation}
Here, the jump symbol $\left\llbracket\cdot\right\rrbracket$ is defined on the interface $u$ by $\llbracket y \rrbracket := y_{\text{in}} - y_{\text{out}}$, where  $y_{ \text{in}} := \tr_{\text{in}}(y\vert_{\Din})$ and $y_{ \text{out}} := \tr_{\text{out}}(y \vert_{\Dout})$, and $\tr_{\text{in}}:\Din \rightarrow u$, $\tr_{\text{out}}:\Dout \rightarrow u$ are trace operators. We will often use the notation $y(\xi)=y(\cdot,\xi).$

With the tracking-type objective functional $J^\text{obj}$
the model is fitted to data measurements $\bar{y}$. Further, $J^\text{reg}$ in (\ref{eq:problem}) is a perimeter regularization and is often required for well-posedness; see for instance \cite[Section 1.1]{Sokolowski1991}.

\subsection{Shape differentiability and bounded variance}
\label{ssec:model-analytical}
In this section, we show shape differentiability for the model problem \eqref{eq:problem}-\eqref{eq:jumpconditions} as well the bound on the second moment of the stochastic gradient;  the latter condition is required by \cref{thm:convergence} in order for  \cref{alg:stochastic_descent-manifold} to converge. In this work, we do not verify the remaining assumptions of \cref{thm:convergence}.  Throughout this section, $c$ denotes a generic deterministic constant (not depending on $\xi$).

For a $r>0$, we define the real Hilbert space $H_{\text{av}}^r(D) := \{  v \in H^r(D) | \int_D v \,\dx = 0\}$ and denote its norm by $\lVert \cdot \rVert_{H^r(D)}.$
Recall that for a Banach space ($X$, $\lVert \cdot \rVert_X$) and a measure space $(\Xi, \mathcal{X}, P),$ the Bochner spaces $L^p(\Xi,X)$ and $L^\infty(\Xi,X)$ are defined as the sets of strongly $\mathcal{X}$-measurable functions $y:\Xi \rightarrow X$ such that
\begin{align*}
 \lVert y \rVert_{L^p(\Xi,X)} &:= \left( \int_\Xi \lVert y(\xi)\rVert^p_X \, \d P(\xi) \right)^{1/p},
 \quad \quad \lVert y \rVert_{L^\infty(\Xi, X)}:=\esssup_{\xi \in \Xi} \lVert y(\xi)\rVert_X
\end{align*}
are finite, respectively. The following technical assumptions are in force in this section.
\begin{assumption}\label{assumption1}
The domain $D\subset\R^2$ is assumed to be a bounded Lipschitz domain and $\bar{y} \in H^2(D)$. In addition, the random fields satisfy the following assumptions: (A1) There exist $\kappa^{\min}$, $\kappa^{\max} >0$ such that $\kappa_i(x, \omega) \in [\kappa^{\min},\kappa^{\max}]$ for almost every~$(x,\xi)\in D\times \Xi$ and for $i\in \{\textup{in},\textup{out}\}$ and (A2) $g \in {L^2(\Xi,H^{1/2}(\Go))}.$
\end{assumption}

Existence and uniqueness of solutions to the PDE constraint under these conditions is classical.
\begin{lemma}\label{lemma:existenceuniquenessoptimalcontrol-stochastic}
For almost every~$\xi \in \Xi$ and all $u\in \UU_D$, there exists a unique solution $y(\xi) = y(\cdot,\xi)\in \hH$ to \cref{eq:PDE1}--\cref{eq:jumpconditions}. Moreover, there exists a $C_1>0$ such that for almost every $\xi \in \Xi$,
\begin{align}
 \lVert y(\xi) \rVert_{H^1(D)} \leq C_1  \lVert g(\xi) \rVert_{H^{1/2}(\Go)}.\label{eq:stateequation-apriori-estimate}
\end{align}
\end{lemma}
\begin{proof}
See \cref{appendix-A}.
\end{proof}

We also need the following strong convergence result, which is required for both the proof of shape differentiability of $J$ and of $j$ in~\Cref{thm:shape_differentiability,thm:j-shape-differentiable}, respectively.  
\begin{lemma}
\label{strong-convergence-state-variables}
Let $D_t:= F_t^V(D)$, $\xi\in\Xi$ be a fixed realization, and $y$ be the solution to~\eqref{eq:PDE1}--\eqref{eq:jumpconditions}. Furthermore, we denote by $y^t: D \times \Xi \rightarrow \R, (x,\xi)\mapsto y^t(x,\xi)$ the solution to the \emph{perturbed state equation} 
\begin{equation}
\label{eq:perturbed_equation}
\int_D \kappa^t(\xi)A(t)\nabla y^t\cdot \nabla \psi\, \dx \, = \int_{\partial D} g(\xi)\psi\, \ds, 
\end{equation}
for all $\psi\in H^1(D)$. Then there exists $\tau>0$ and $c > 0$ such that for all $t\in[0,\tau]$
\begin{equation}
\label{eq:strong-convergence}
 \lVert y^t(\xi) - y(\xi) \rVert_{H^1(D)} \leq ct\, \lVert y(\xi) \rVert_{H^1(D)}.
\end{equation}
\end{lemma}

\begin{proof}
See \cref{appendix-A}.
\end{proof}

\begin{theorem}
\label{thm:shape_differentiability}
For almost every $\xi \in \Xi$ and all $u\in \UU_D$, the shape functional  $J^{\textup{obj}}(u,\xi)$ defined in~\eqref{Objective_TrackingType} is shape differentiable. Furthermore, the weak formulation of the shape derivative for a fixed $\xi \in \Xi$ is given by
\begin{equation}
\label{shape_derivative_J_obj}
\begin{split}
d J^{\textup{obj}}(u,\xi)[W] &= \dfrac12\int_D  \divv(W) (y-\bar{y})^2 \, \dx - \int_D(y-\bar{y})\nabla\bar{y}\cdot W \, \dx\\ 
 &\quad\quad + \int_D \kappa (\divv(W)\id - \nabla W - \nabla W^T)\nabla y \cdot \nabla p \, \dx, 
\end{split}
\end{equation}  
where $\kappa = \kappa(\xi)$, $y=y(\xi)\in \hH$ is the weak solution of~\eqref{eq:PDE1}--\eqref{eq:jumpconditions}, and $p=p(\xi)\in \hH$ solves (with $\kappa=\kappa(\xi)$ and $y = y(\xi)$) the adjoint equation 
\begin{equation}\label{eq:adjoint-weak}
\int_D \kappa \nabla \varphi \cdot \nabla p\, \dx = - \int_D (y-\bar{y})\varphi \, \dx \quad  \forall \varphi \in \hH.
\end{equation} 
\end{theorem}

\begin{proof}
See \cref{appendix-B}.
\end{proof}

Solvability of the adjoint equation \eqref{eq:adjoint-weak} is needed for the proof of \cref{thm:shape_differentiability}.
\begin{corollary}\label{lemma:existenceuniquenessadjoint-stochastic}
For almost every $\xi \in \Xi$ and all $u\in \UU_D$, there exists a unique solution $p(\xi)=p(\cdot,\xi)\in \hH$ to \eqref{eq:adjoint-weak}. Moreover, there exists a constant $C_2>0$ such that for almost every $\xi \in \Xi$,
\begin{equation}\label{eq:adjointequation-apriori-estimate}
 \lVert p(\xi) \rVert_{H^1(D)} \leq C_2 \lVert y(\xi) - \bar{y} \rVert_{L^2(D)}.
\end{equation}
\end{corollary}
\begin{proof}
See \cref{appendix-A}.
\end{proof}

Clearly, the perimeter regularization is shape differentiable (see e.g.~\cite[Section 3.3]{Sokolowski1991}). With $\iota:=\text{\text{div}}_{u}(\n)$ denoting the mean curvature of $u$, the expression of the shape derivative is given by
 \begin{equation}
\label{sd_j2}
d J^{\text{reg}}(u)[W] = \int_{ u} \iota\, W \cdot \n \,\ds.
\end{equation}

\begin{theorem}
\label{thm:j-shape-differentiable}
The function $j$ is shape differentiable for all $u\in \UU_D$.
\end{theorem}

\begin{proof}
We will verify the conditions of \cref{lemma:differentiability-expectation}. To that end, let $V \in C_0^\infty(D, \R^2)$ be an arbitrary vector field and let the perturbed shape be given by $u_t:=F_t^V(u)$. We observe the quantity $ R_t^V(\xi) :=(J(u_t,\xi) - J(u,\xi))/t$ and $y^t$ the solution of~\eqref{eq:perturbed_equation}. Now, using the transformations $\eta(t)\coloneqq \det(DF_t^V) $ and $\tilde{\eta}(t) \coloneqq \eta(t) \lVert (DF_t^V)^{-*} \cdot \n \rVert_2$ (cf.~\cite[p.~482]{Delfour-Zolesio-2001},~\cite[p.79]{Sokolowski1991}, respectively)
\begin{align*}
 tR_t^V(\xi)  &= \frac{1}{2} \int_{D_t} (y_t(\xi) - \bar{y})^2 \dx -  \frac{1}{2} \int_{D} (y(\xi)  - \bar{y})^2\dx +\nu\int_{u_t} \ds - \nu\int_u \ds\\
 &=\frac{1}{2} \int_D \eta(t) (y^t(\xi) - \bar{y}^t)^2  \dx - \frac{1}{2} \int_D (y(\xi)-\bar{y})^2  \dx + \nu \int_u (\tilde{\eta}(t) -1)\ds .
\end{align*}
where $y_t\in H^1(D_t)$ is the solution of the state equation when we replace $u$ by $u_t$. Thanks to~\cite[p.~526]{Delfour-Zolesio-2001}, we know that there exists a $\tau_1>0$ such that $\eta(t)$ is bounded for all $t\in[0,\tau_1]$. Therefore,
\begin{align}
 \nonumber t R_t^V(\xi) &\leq  c \, \frac12 \int_{D}(y^t(\xi) - \bar{y}^t)^2 \dx -  \frac{1}{2} \int_{D} (y(\xi)  - \bar{y})^2\dx +\nu\int_{u} (\tilde{\eta}(t)-1) \, \ds\\
\nonumber  &\leq c\int_{D}[(y^t(\xi))^2 + (\bar{y}^t)^2] \, \dx +\nu\int_{u} (\tilde{\eta}(t)-1) \, \ds\\
\label{eq:bound_Rt}& = c \,  [\lVert y^t(\xi)\rVert^2_{L^2(D)} + \lVert \bar{y}^t\rVert^2_{L^2(D)}] +\nu\int_{u} (\tilde{\eta}(t)-1) \, \ds.
\end{align}
Using~\cref{strong-convergence-state-variables} and the inverse triangular inequality, we get that there exists $\tau_2$ small enough such that by \eqref{eq:stateequation-apriori-estimate},
\begin{equation}
\label{eq:bound_yt}
\lVert y^t(\xi) \rVert_{H^1(D)} \leq (ct + 1)\lVert y(\xi)\rVert_{H^1(D)} \leq (ct + 1)C_1\lVert g(\xi)\rVert_{H^{1/2}(\partial D)} \quad \forall t\in[0,\tau_2].
\end{equation}

Now, since $\bar{y}\in L^2(D)$, we know by \cite[Lemma~2.16]{Sturm2014} that $\lim_{t\to 0}\lVert\bar{y}^t - \bar{y}\rVert_{L^2(D)}=0$. Thus, there exists $\tau_3$ small enough such that
\begin{equation}
\label{eq:bound_ybar}
\lVert \bar{y}^t\rVert_{L^2(D)} \leq 1 + \lVert \bar{y}\rVert_{L^2(D)}\quad\text{for all } t\in[0,\tau_3].
\end{equation}
Finally, by~\cite[Lemma~2.49]{Sokolowski1991} we know $\tilde{\eta}(t)$ is differentiable, therefore continuous, for $t\in[0,\tau_4]$ and $\tau_4$ small enough. Then, there exists $\widetilde{C}>0$ such that $|\tilde{\eta}(t)|< \widetilde{C}$ for all $t\in[0,\tau_4]$. 
Therefore, by \eqref{eq:bound_yt} and \eqref{eq:bound_ybar}, \eqref{eq:bound_Rt} becomes
\[
t R_t^V(\xi) \leq c\left[\left((c\tau +1)C_1\lVert g(\xi)\rVert_{H^{1/2}(\partial D)}\right)^2 +(1+\lVert \bar{y}\rVert_{L^2(D)})^2\right] +\nu\int_{u} (\tilde{C}-1) \, \ds  =:C(\xi).
\]
with $\tau:=\min\{\tau_1,\tau_2,\tau_3,\tau_4\}$. Thus, we have obtained a dominating function that is $\PP$-integrable by \cref{assumption1}, (A2). By \cref{lemma:differentiability-expectation}, we have the conclusion.
\end{proof}

We now show that the second moment of the stochastic gradient is bounded as required in \cref{thm:convergence}. Recall that $v = v(u,\xi)$ is generated by the solution $V \in H_0^1(D,\R^2)$ to \eqref{deformatio_equation} with $v=\textup{tr}(V)\cdot \n$. The assumption that the boundary of $D$ is smooth is used to obtain higher regularity of the state and adjoint solutions. 

\begin{lemma}
\label{lemma:bounded-variance}
Assume that the boundary of $D$ is of class $C^2$. Then there exists a constant $M>0$ such that for all $u \in \UU_D$,
 \begin{equation}
  \label{eq:bounded-variance-model-problem}
 \EE[\lVert v(u,\xi)\rVert^2] \leq M.
 \end{equation}

\end{lemma}
\begin{proof}
Let $u\in \UU_D$ be arbitrary but fixed. We denote the norm on the piecewise Sobolev space 
$PH^k(D):=\{ v=v_{\text{in}}\mathbbm{1}_{D_{\text{in}}} + v_{\text{out}}\mathbbm{1}_{D_{\text{in}}} \, \vert \, v_\text{in} \in H^k(D_{\text{in}}), v_\text{out} \in H^k(D_{\text{out}})\}$
by $\lVert v \rVert_{PH^k(D)}^2:=\lVert v \rVert_{H^k(D_{\text{in}})}^2 + \lVert v \rVert_{H^k(D_{\text{out}})}^2$. 
\paragraph{Part 1}
Using standard arguments adapted to the function space $H_{\text{av}}^1(D)$ (see e.g.~Section 3.2 from~\cite{Ito-Kunisch-Peichl}), it is possible to show that $y|_{D_{i}}, p|_{D_{i}} \in H^2(D_{i})$ for $i \in \{\text{in, out}\}$. We use the fact that the boundary of the domain $D$ is smooth enough, so by \cite[Theorem 5.2.1]{Costabel2010}, we have for a fixed $\xi\in \Xi$ and $y=y(\xi), p=p(\xi)$ the following a priori bounds
\begin{equation}
\label{eq:H2-bounds}
\begin{aligned}
 \lVert y \rVert_{PH^2(D)} &\leq c (\lVert g\rVert_{H^{1/2}(\partial D)} + \lVert y \rVert_{H^1(D)}),\\
  \lVert p \rVert_{PH^2(D)} &\leq c (\lVert y - \bar{y}\rVert_{L^2(D)} + \lVert p \rVert_{H^1(D)}).
\end{aligned}
\end{equation}

\paragraph{Part 2} We now show that there exists $C \in L^2(\Xi)$ such that for all $W \in H_0^1(D,\R^2)$,
\begin{equation}
 \label{eq:bound-shape-derivative-second-moment}
 d J(u,\xi)[W] \leq C(\xi) \lVert W \rVert_{H^1(D,\R^2)}.
\end{equation}

We use the fact that $H^1(D)$ is compactly embedded in $L^4(D)$ (cf.~\cite[p.~345]{Alt2012}).  Notice that 
\begin{align*}
\lVert \nabla y \rVert_{L^4(D)} &= \lVert \nabla y \rVert_{L^4(D_{\text{in}})} + \lVert \nabla y \rVert_{L^4(D_{\text{out}})} \leq c\left(\lVert \nabla y \rVert_{H^1(D_{\text{in}})} + \lVert \nabla y \rVert_{H^1(D_{\text{out}})}\right)\\
&\leq c \left(\lVert y \rVert_{H^2 (D_{\text{in}})} + \lVert y \rVert_{H^2 (D_{\text{out}})}\right) \leq c(\lVert y \rVert_{PH^2 (D)}).
\end{align*}
Now, by \eqref{shape_derivative_J_obj}, we obtain by elementary inequalities and the successive invocation of the H\"older's inequality that
\begin{align*}
 &|d J^{\text{obj}}(u,\xi)[W]| \\
 &\leq \tfrac{1}{2} \lVert \divv(W) (y-\bar{y})^2 \rVert_{L^1(D)}+\lVert(y-\bar{y})\nabla\bar{y}\cdot W \rVert_{L^1(D)} \\
 &\qquad +\lVert \kappa (\text{div}(W) \id - \nabla W - \nabla W^T)\nabla y \cdot \nabla p \rVert_{L^1(D)} \\
 & \leq c ( \lVert \text{div}(W)\rVert_{L^2(D)} \lVert y - \bar{y}\rVert_{L^4(D)}^2 + \lVert W \rVert_{L^2(D,\R^2)} \lVert y-\bar{y} \rVert_{L^4(D)} \lVert \nabla \bar{y} \rVert_{L^4(D)} \\
 & \qquad + \lVert \kappa \rVert_{L^\infty(\Xi)} ( \lVert \text{div}(W) \rVert_{L^2(D)} +| W |_{H^1(D,\R^2)}) \lVert \nabla y \rVert_{L^4(D)} \lVert \nabla p \rVert_{L^4(D)}) \\
 & \leq c (\lVert y - \bar{y}\rVert_{L^4(D)}^2 + \lVert y-\bar{y} \rVert_{L^4(D)} \lVert \nabla \bar{y} \rVert_{L^4(D)}  + \lVert \nabla y \rVert_{L^4(D)} \lVert \nabla p \rVert_{L^4(D)}) \lVert W \rVert_{H^1(D,\R^2)}\\
 & \leq c (\lVert y - \bar{y}\rVert_{H^1(D)}^2 + \lVert y-\bar{y} \rVert_{H^1(D)} \lVert \nabla \bar{y} \rVert_{H^1(D)} +\lVert y \rVert_{PH^2(\tilde{D})} \lVert p\rVert_{PH^2(\tilde{D})} ) \lVert W \rVert_{H^1(D,\R^2)}.
\end{align*}
Using \eqref{eq:stateequation-apriori-estimate}, \eqref{eq:adjointequation-apriori-estimate}, \eqref{eq:H2-bounds}, as well as the assumption of measurability from \cref{assumption1}, (A2), we obtain a $C \in L^2(\Xi)$ such that $d J^{\textup{obj}}(u,\xi)[W] \leq C(\xi) \lVert W \rVert_{H^1(D,\R^2)}$. 

The boundary term \eqref{sd_j2} can be bounded in a similar way by using the trace theorem \cite[p.~279]{Alt2012} and observing
\begin{align*}
|d J^{\text{reg}}(u)[W]| & \leq \lVert W \rVert_{L^2(u, \R^2)} \lVert \iota \n \rVert_{L^2(u, \R^2)} \leq c \lVert W \rVert_{H^1(D, \R^2)} \lVert \iota \n \rVert_{L^2(u, \R^2)}.
\end{align*}
Finally, we have obtained \eqref{eq:bound-shape-derivative-second-moment}. 

\paragraph{Part 3} Now, by coercivity of $a^{\textup{elas}}(\cdot,\cdot)$ and \cref{remark:convergence-manifold-implies-convergence-holdall}, for $V \in H_0^1(D, \R^2)$ satisfying \eqref{deformatio_equation},
\begin{equation}
\label{eq:variance-term-coercivity-bilinear-form}
 k \lVert V \rVert_{H^1(D,\R^2)}^2 \leq a^{\textup{elas}}( V, V) = d J(u,\xi)[ V] \leq C(\xi) \lVert  V \rVert_{H^1(D,\R^2)}
\end{equation}
so in particular, $\lVert V \rVert_{H^1(D,\R^2)} \leq C(\xi)/k$, implying $\EE[\lVert V \rVert_{H^1(D,\R^2)}^2]$ is finite since we have $C \in L^2(\Xi)$. By boundedness of $a^{\textup{elas}}(\cdot,\cdot)$, we have
\begin{equation}
\label{eq:variance-term-boundedness-bilinear-form}
\lVert v(u,\xi) \rVert^2 = G(v,v) = a^{\textup{elas}}(V,V) \leq K \lVert V\rVert^2_{H^1(D,\R^2)}.
\end{equation}
Thus, combining \eqref{eq:variance-term-coercivity-bilinear-form} and \eqref{eq:variance-term-boundedness-bilinear-form}, we have 
$$\EE[\lVert v(u,\xi) \rVert^2] \leq K \EE[\lVert V \rVert_{H^1(D,\R^2)}^2] < \infty,$$
so \eqref{eq:bounded-variance-model-problem} is satisfied. Since $u$ was chosen to be arbitrary, we have the conclusion.
\end{proof}

\section{Numerical results}
\label{sec:Numerics}
In this section, we present results of numerical experiments to demonstrate the performance of \cref{alg:stochastic_descent-manifold}.
It is worth mentioning that interface identification problems are highly ill-posed; therefore, their numerical solution is extremely challenging. Most of the previous work in this direction has dealt with the identification of convex and/or singles shapes (without a stochastic model). To demonstrate the performance of the algorithm, we are including an example in which we identify multiple nonconvex shapes.
In~\cref{subsec:simple_shapes}, we present the numerical solution of the model problem from \cref{sec:model}. Additionally, we verify the Lipschitz gradient assumption numerically. In~\cref{subsec:multiple_shapes}, we show that the algorithm can also be applied to more realistic applications involving the identification of multiple shapes. 

The numerical solution of shape optimization problems has many challenges. For methods relying on mesh deformation, one challenge is to keep the mesh quality under control. We have discussed this issue in more detail in \cite{Geiersbach2019}. As in \cite{Geiersbach2019}, we choose the Lam\'e parameters from \eqref{eq:linear-elasticity-bilinear-form} to be $\lambda=0$ and solve a Poisson problem to compute $\mu$; we also restrict test functions in the assembly of the shape derivative as described in \cite{Geiersbach2019}.

To update the shapes according to \cref{alg:stochastic_descent-manifold}, we need to compute the exponential map.
This computation is prohibitively expensive in the most applications because a calculus of variations problem must be solved or the Christoffel symbols need be known. 
Therefore, we approximate it using a retraction. 
We use the following twice differentiable\footnote{The chosen retraction (\ref{retraction}) is obviously twice differentiable as required by our theory. The second derivative is given by the zero element of the tangent space.}  retraction as in \cite{SchulzWelker}: 
\begin{equation}
\label{retraction}
\mathcal{R}_u\colon T_u\mathcal{U} \to \mathcal{U},\, 
v \mapsto \mathcal{R}_u(v)\coloneqq u+v. 
\end{equation}
We note that the retraction is only a local approximation; for large vector fields, the image of this function may no longer belong to $B_e$. This retraction is closely related to the perturbation of the identity, which is defined for vector fields on the domain $D$.  Given a starting shape $u_{n+1}$ in the $n$-{th} iteration of \cref{alg:stochastic_descent-manifold}, the perturbation of the identity acting on the domain $D$ in the direction $V_n$, where $V_n$ solves~\eqref{deformatio_equation}, gives
\[
D(u_{n+1}) = \{x \in D \, | \, x = x_n + t_n V_n \}.
\]
As vector fields induced from solving~\eqref{deformatio_equation} have less regularity than is required on the manifold, we remark that the shape $u_{n+1}$ resulting from this update could leave the manifold $B_e$. To summarize, either large or less smooth vector fields can contribute to the iterate $u_{n+1}$ leaving the manifold. One indication that the iterate has left the manifold would be that the curve $u_{n+1}$ develops corners; however, since we discretize the curve this is not able to be observed numerically. Another possibility is that the curve $u_{n+1}$ self-intersects. One way to avoid this behavior is by preventing the underlying mesh to break (meaning elements from the finite element discretization overlap). We avoided broken meshes as long as the step-size was not chosen to be too large.
 
In the following experiments, we assume the random parameters are distributed according to $\mathcal{N}(\rho,\sigma,a,b)$, which is the truncated normal distribution with parameters $\rho$ and $\sigma$ and bounds $a,b$. The details of the parameters will be given in each experiment. 
The experiments were performed in a CPU Intel Core i7-7500 with 2.7 GHz and 15GB RAM.

\subsection{Single shapes}
\label{section:single_shapes}
\label{subsec:simple_shapes}
This experiment can be understood as the identification of a human lung, where the target $\bar{y}$ is to be obtained using electrical impedance tomography. 
We set $D=[-1,0]\times[-0.5,0.5]$ and the shape to be identified is shown in~\Cref{fig:simple_data} (left). For the numerical experiments, we make a simplification and consider the boundary data $g \equiv 10$ to be deterministic. 
\begin{figure}
\begin{center}
\includegraphics[width = 4.5cm,height= 4cm]{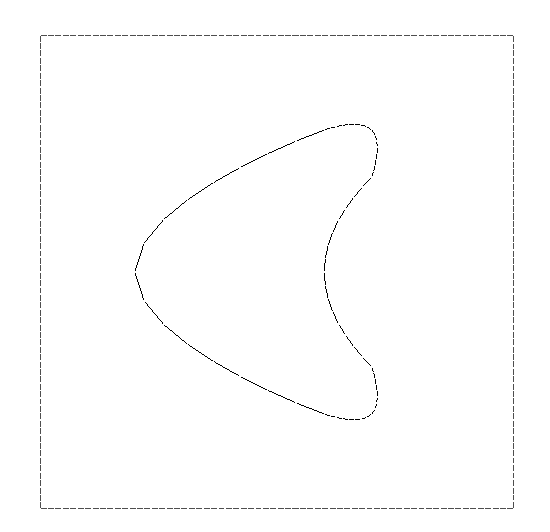}
\qquad
\includegraphics[width = 5.5cm,height= 4cm]{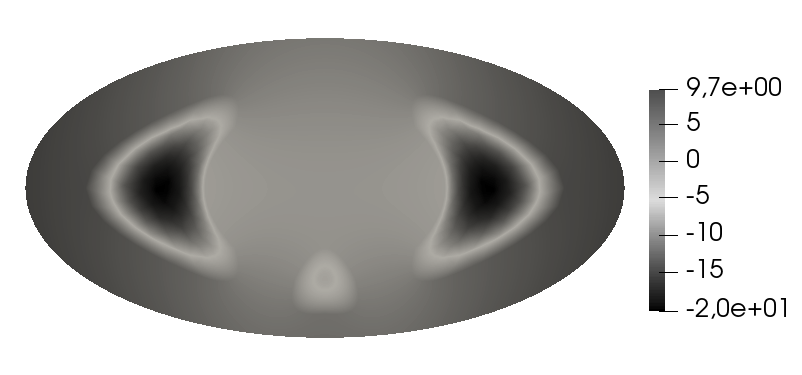}
\end{center}
\caption{Single shape: target shape (left) and input data $\bar{y}$ (right).}
\label{fig:simple_data}
\end{figure}
On $D$, we generate a triangular mesh of 3006 nodes and 6074 elements, and solve the state equation~\eqref{eq:PDE1}--\eqref{eq:jumpconditions} with the parameters $\kappaTrunkdet = 1$ and $\kappaLungsdet = 0.005$. The solution of this equation corresponds to $\bar{y}$  and is depicted in~\Cref{fig:simple_data} (right).

For the stochastic model, we consider conductivity parameters that follow the distributions: $\kappaTrunk \sim \mathcal{N}(\kappaTrunkdet,10^{-4},0.7,1.3)$ and $\kappaLungs \sim \mathcal{N}(\kappaLungsdet,10^{-4},2.5 \cdot 10^{-3},7.5 \cdot 10^{-3})$. The parameter for the perimeter regularization is fixed to $\nu = 10^{-6}$. 
Since the retraction is only defined locally, we need to choose step-sizes that are small enough to ensure that iterates do not leave the manifold.
After tuning, we have obtained that a reasonable choice for the step size rule is 
$t_n= \tfrac{0.016}{n}$, which was obtained after tuning.
We choose  $\mu_{\min} = 5$ and $\mu_{\max}= 17$ for the computation of $\mu_n$ as discussed in~\cite{Geiersbach2019}. We let the algorithm iterate 500 times and the initial, intermediate and final shapes obtained are depicted in~\Cref{fig:iterations_simple}. The experiment took 6.5828 minutes. 

\begin{figure}
\begin{center}
\includegraphics[scale=0.13]{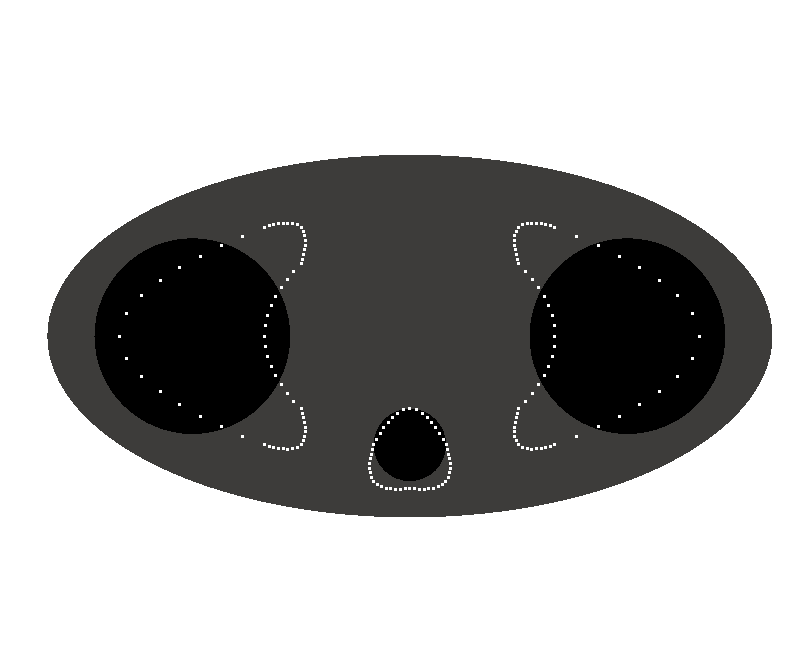}
\quad
\includegraphics[scale=0.13]{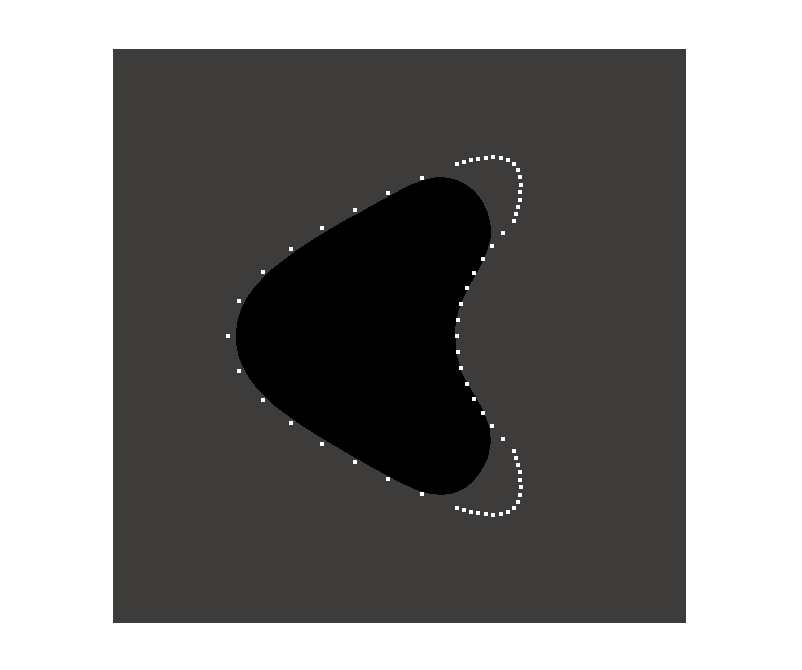}
\quad
\includegraphics[scale=0.13]{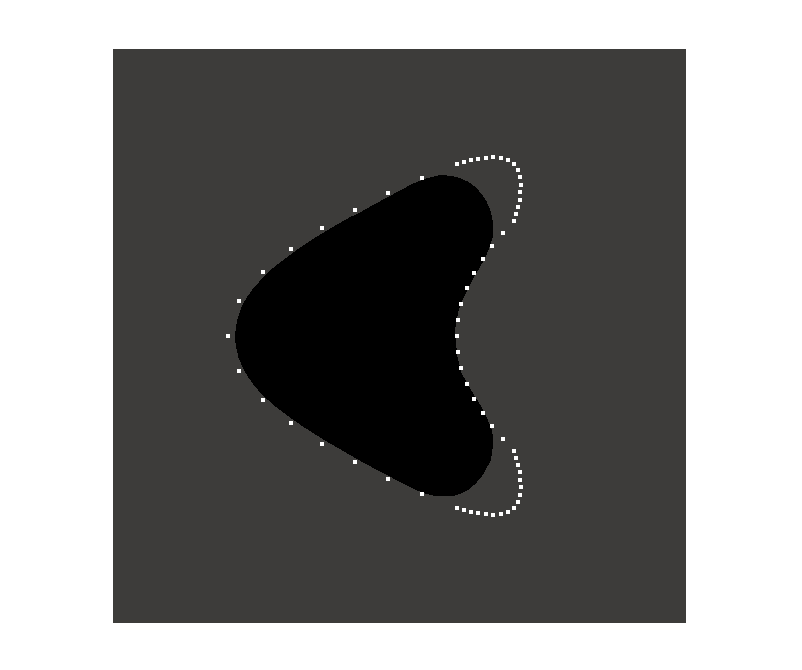}
\end{center}
\caption{Single shapes: initial (left), 250th (middle) and 500th (right).}
\label{fig:iterations_simple}
\end{figure}

The behavior of the decreasing of the objective function as depicted in~\Cref{fig:decreasing_simple} (left) demonstrates the typical behavior of the stochastic gradient method. According to \cref{remark:convergence-manifold-implies-convergence-holdall}, we expect the $H^1$ norm of the deformation field to converge to zero, which we can observe in \Cref{fig:decreasing_simple} (right).  We emphasize that oscillations in the plots come from the fact that we are using single estimates  $J(u_n,\xi_n)$ for the function value $j(u_n)$ along with the fact that the stochastic gradient method is not a descent method; for this reason, we observe oscillations in $\lVert V_n \rVert_{H^1(D,\R^2)}$. 
  
\begin{figure}
\begin{center}
\includegraphics[scale=0.5]{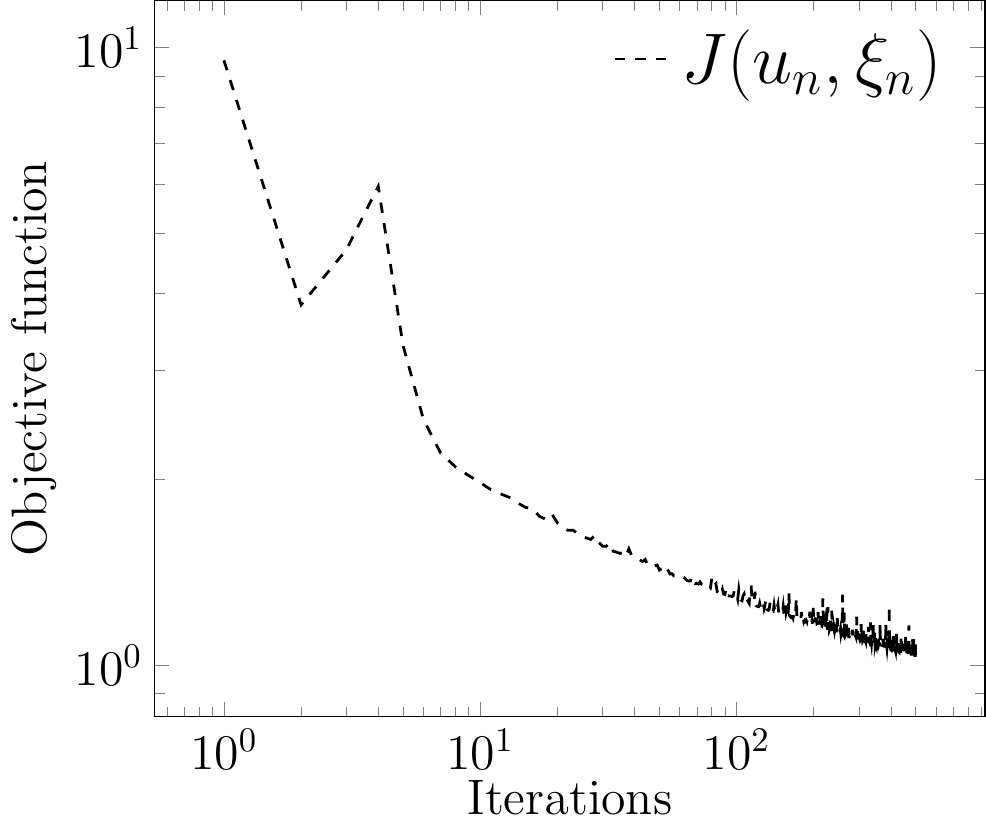}
\qquad
\includegraphics[scale=0.5]{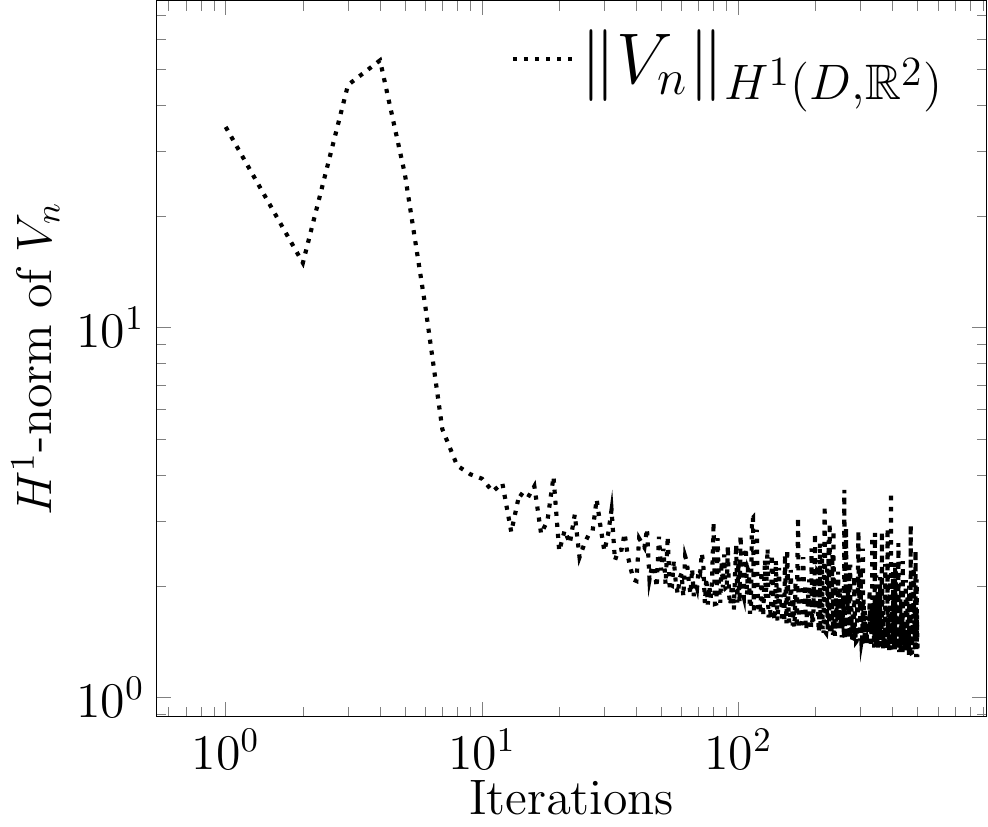}
\end{center}
\caption{Single shapes: behavior of the objective (left) and $H^1$-norm of the deformation field (right) in a log/log scale.}
\label{fig:decreasing_simple}
\end{figure} 

\subsection*{Numerical verification of assumptions for convergence}
In this test, we numerically approximate the Lipschitz constant from the condition \eqref{eq:Lipschitz-condition} for the gradient of $j(u)$. While this cannot provide us with the value for the constant over all shapes contained in $D$, this experiments gives us insight into its magnitude along the sample path. As should be evident by the calculations presented in the proof for \cref{thm:shape_differentiability}, a rigorous proof of higher-order derivatives would be quite lengthy. 

As in \cite{LuftWelker}, we approximate the distance $\d(u,\tilde{u})$ between between two shapes $u,\tilde{u}$ by
${\d}^{\text{approx}}(u,\tilde{u}):= \int_{u} \max_{y\in \tilde{u}}\|x-y\|_2\, \dx.$ For the bound on the gradient of $j$, we use the fact that $P_{1,0}$ is an isometry and the definition of $\nabla j$ to get the second inequality followed by Jensen's inequality and \cref{Steklov_identity} to get
\begin{align*}
 \lVert P_{1,0} \nabla j(u_{n}) - \nabla j(u_1) \rVert 
 &\leq \lVert P_{1,0} \nabla j(u_{n})\rVert + \lVert \nabla j(u_1)\rVert\\
 &\leq \lVert \EE[\nabla J(u_{n},\xi)]\rVert + \lVert \EE[\nabla J(u_{1},\xi)]\rVert\\
&\leq \EE[\sqrt{dJ(u_{n},\xi)[V_{n}]}] + \EE[\sqrt{dJ(u_{1},\xi)[V_{1}]}].
\end{align*}
We use the approximation 
\begin{equation}
\label{eq:approximation}
\EE[\sqrt{dJ(u_{n},\xi)[V_{n}]}] \approx \frac{1}{m}\sum_{j=1}^m \sqrt{dJ(u_{n},\xi_{n}^{j})[V_{n}]},
\end{equation}
where $m=100$ new i.i.d.~samples $\xi_{i}^j$, distributed as described in \cref{section:single_shapes}, were drawn at iteration $n$ for $j=1, \dots, m$. For all iterations, we compute the quotient
\begin{equation}
\label{eq:L_j}
L_n \coloneqq \frac{\frac{1}{m}\displaystyle\sum_{j=1}^m \left(\sqrt{dJ(u_{n},\xi_{n}^{j})[V_{n}]} + \sqrt{dJ(u_{1},\xi_1^{j})[V_{1}]} \right)}{\d^{\text{approx}}(u_n,u_1)}
\end{equation}
and we show in \Cref{fig:Lipschitz_j} that for every iteration this value is bounded. 

\begin{figure}
\centering
\includegraphics[scale=0.5]{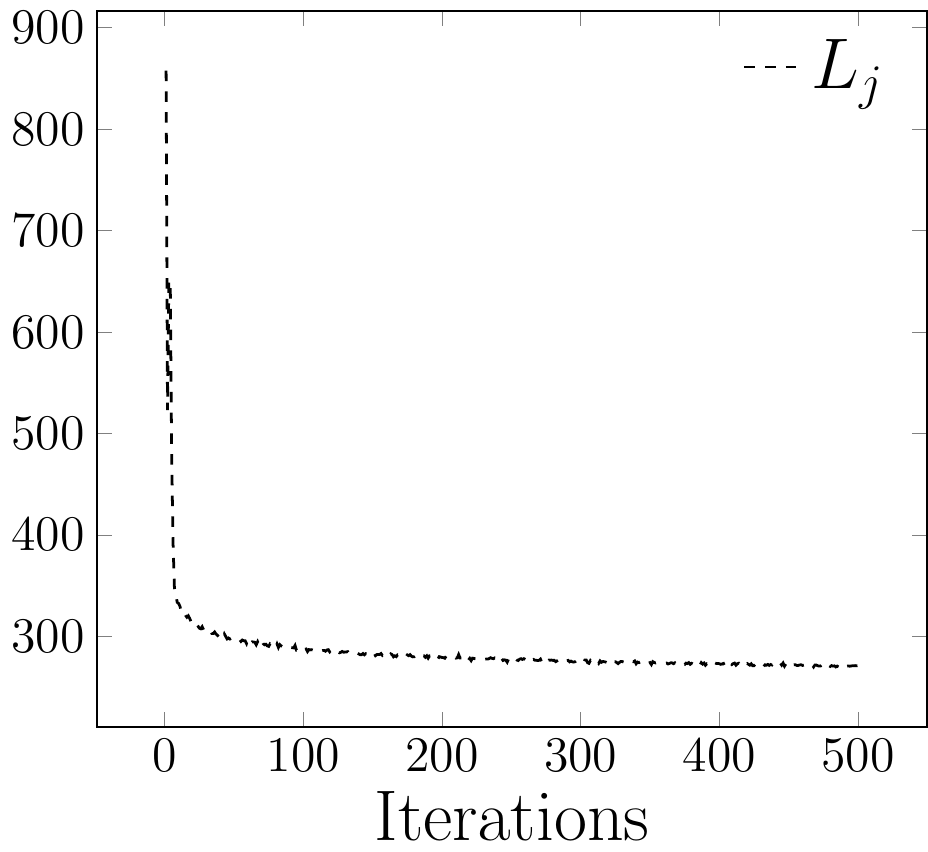}
\caption{Values obtained for the expression~\eqref{eq:L_j}.}
\label{fig:Lipschitz_j}
\end{figure} 

\subsection{Multiple shapes}
\label{subsec:multiple_shapes}

The main objective of this experiment is to show that the algorithm can also be applied to more realistic problems. In this case, we consider an ellipsoidal domain centered in the origin with major axis of length 1 and minor axis of length 0.5, containing three nonintersecting shapes to be identified, which may be understood as the cross-section of the human body containing the heart and lungs. The target shapes are depicted in~\Cref{fig:multiple-shapes} (left). 

\begin{figure}
\begin{center}
\includegraphics[height=3cm, width = 4.5cm]{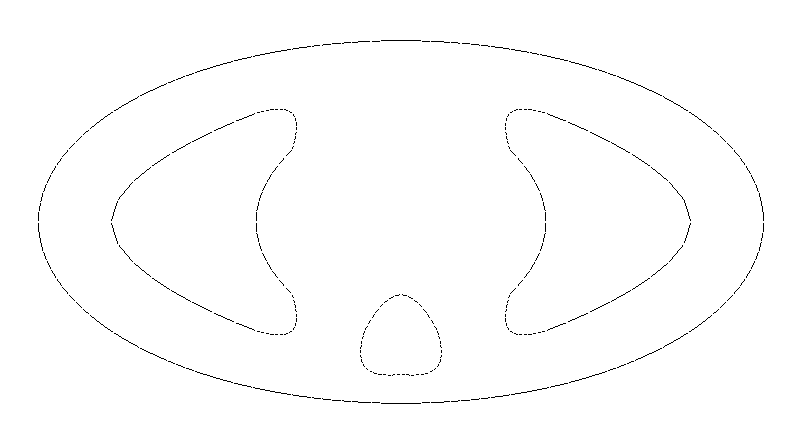}
\quad
\includegraphics[height=3cm, width = 5.5cm]{ybar.png}
\end{center}
\caption{Multiple shapes: target shapes (left) and input data $\bar{y}$ (right)}
\label{fig:multiple-shapes}
\end{figure}

The values of $\bar{y}$ were obtained as in the previous experiment, using the same values for $\kappaLungsdet$ and $\kappaTrunkdet$ and using $\kappaHeartdet=0.015$. This solution is depicted in~\Cref{fig:multiple-shapes} (right).  We mention that working with multiple shapes has its own theoretical difficulties. For one, the shape space over which one optimizes is a product space of $\UU$. One approach to solve a problem with multiples shapes would be to partition the domain $D$ into subdomains containing one shape each. This would however presume that we have prior knowledge as to the placement and number of shapes to be identified. Here, we assume we know the number of target shapes and show that our approach works even with multiple shapes.

The random parameters are assumed to be distributed as follows: 
\begin{itemize}
\item $\kappaHeart \sim \mathcal{N}(\kappaHeartdet,10^{-3},0.01,0.02)$
\item $\kappaLungs \sim \mathcal{N}(\kappaLungsdet,10^{-3},2.5 \cdot 10^{-3},7.5\cdot 10^{-3})$ 
\item $\kappaTrunk \sim \mathcal{N}(\kappaTrunkdet,10^{-3},0.7,1.3)$
\end{itemize}
The value of the parameter for the perimeter regularization is $\nu = 10^{-6}$. For the step-size rule we use $t_n = \tfrac{0.15}{n}$, and $\mu_{\min} = 5$ and $\mu_{\max}= 17$ are chosen for the Lam\'e parameter problem. The mesh has 3210 nodes and 6578 elements. We let the algorithm run 300 iterations. The initial, intermediate, and final shapes are shown in~\Cref{fig:iterations_multiple}. The experiment took 1.4130 minutes.

\begin{figure}
\begin{center}
\includegraphics[scale=0.13]{movie.0000.png}
\quad
\includegraphics[scale=0.13]{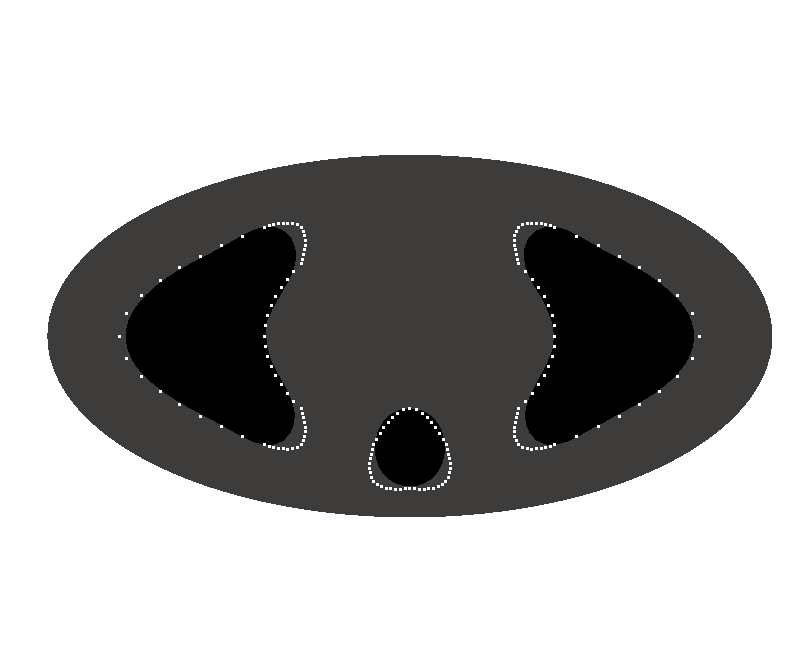}
\quad
\includegraphics[scale=0.13]{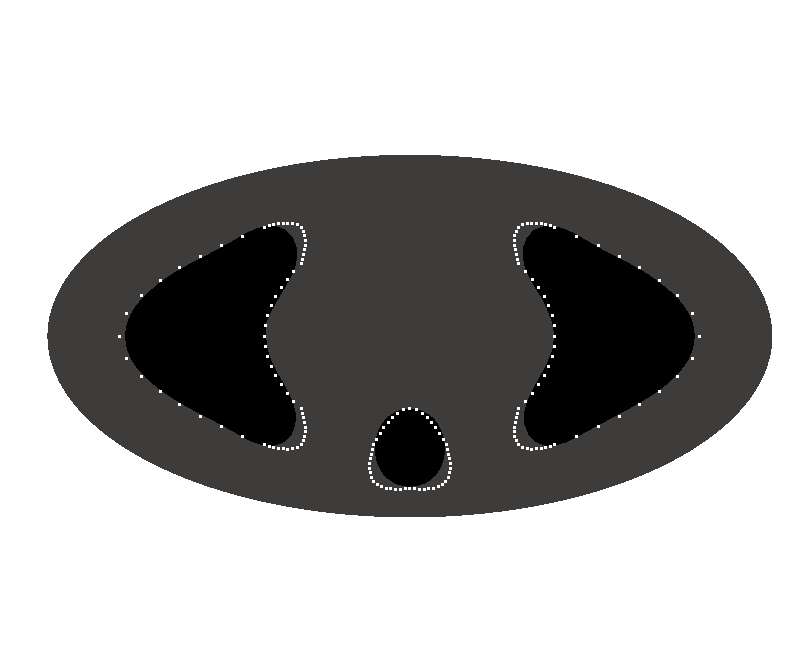}
\end{center}
\caption{Multiple shapes: initial (left), 150th (middle) and 300th (right)}
\label{fig:iterations_multiple}
\end{figure}

In~\cref{fig:multiple_decay} (left), we show the behavior of the objective function, in which we can appreciate the typical behavior of the stochastic gradient algorithm.  Again, we can observe the $H^1$-norm of the deformation field tending to zero in \cref{fig:multiple_decay} (right).

\begin{figure}
\begin{center}
\includegraphics[scale=0.5]{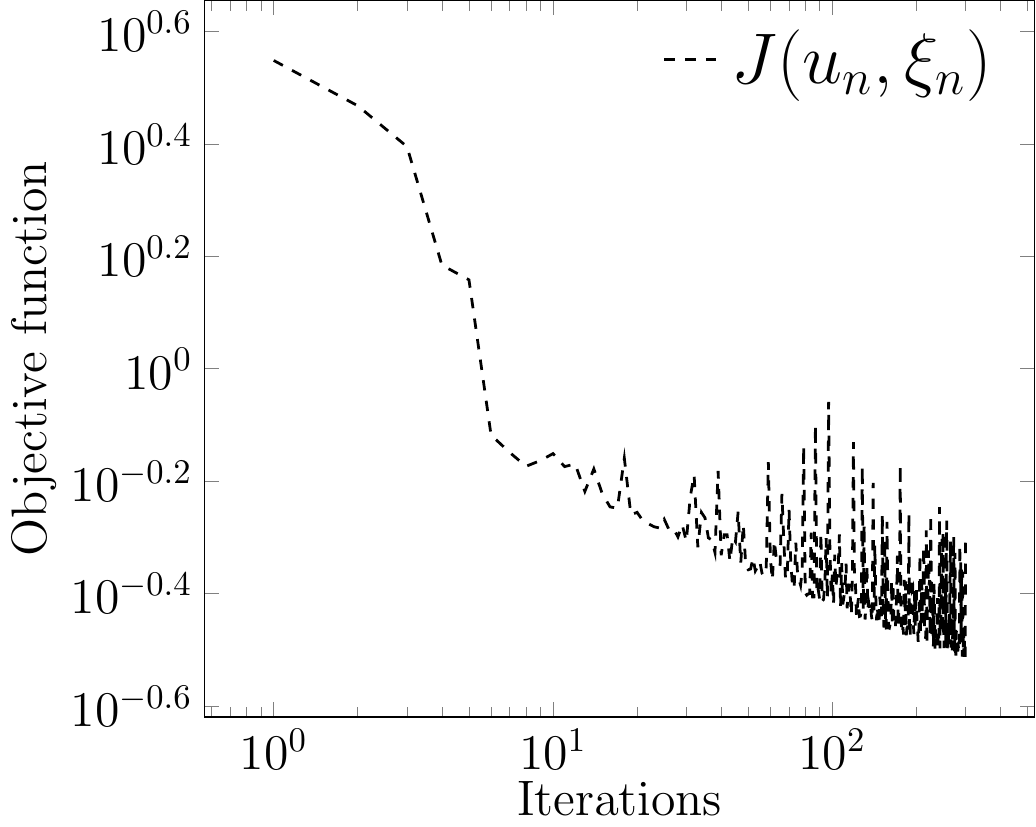}
\quad
\includegraphics[scale=0.5]{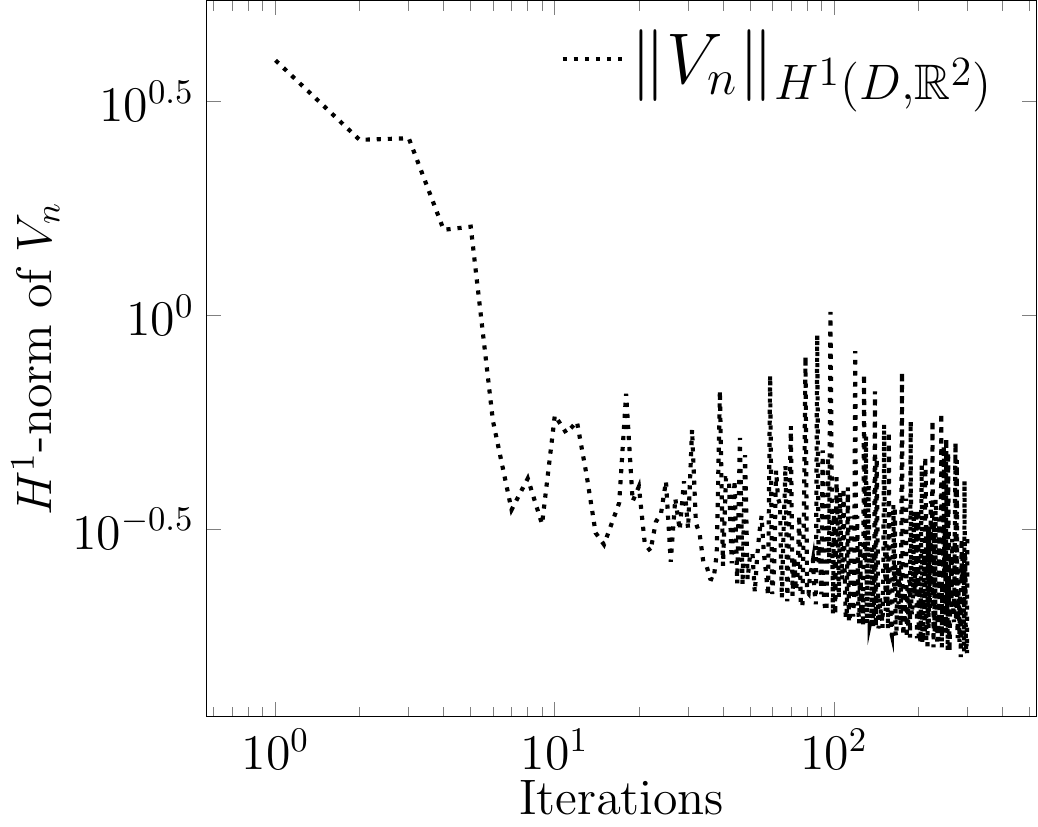}
\end{center}
\caption{Multiple shapes: behavior of the objective (left) and $H^1$-norm of the deformation field (right)}
\label{fig:multiple_decay}
\end{figure} 
 
\section{Conclusion}\label{sec:conclusion}
In this paper, we extended the classical stochastic gradient method to a novel approach for solving stochastic shape optimization problems on infinite-dimensional manifolds. Our work combines three research areas: stochastic optimization, shape optimization and infinite-dimensional differential geometry. We show convergence of the proposed method with the classical Robbins-Monro step-size rule. We introduced a model stochastic shape optimization problem based on interface identification, where parameters in the underlying PDE are subject to uncertainty. For this problem, we show shape differentiablility and the bound on the second moment of the stochastic gradient, which are necessary conditions for the convergence of the algorithm.
The modeling of uncertainty in shape optimization allows for more robust solutions in applications where parameters and inputs are not assumed to be known.

Two numerical experiments for the model problem were presented in the paper.   We observed the behavior of the stochastic gradient method in the form of the objective function and the gradient field, which on average decayed with the number of iterations. Additionally, we showed a simulation with the identification of multiple shapes, showing that the method can be applied for more complex models.

Since the connection of the above-mentioned three research areas is quite new, the results of this paper leave space for future research. In particular, there are a few open questions from differential geometry that are outside the scope of the paper but that came up while formulating our theory. It is still unclear whether \cref{assump:manifold} is satisfied for the manifold used in our application. In particular, we require connectivity and the existence of a bounded injectivity radius of the shape space under consideration. Additionally, while the shapes in our model problem are contained in a bounded domain, it is unclear under what conditions the iterates generated by the algorithm remain in a bounded set on the manifold as required by our theory. While we did not investigate higher-order shape differentiability, we note that convergence of the algorithm to stationary points is generally only possible with additional regularity. Finally, the choice of the step-size rule in stochastic approximation is still an active area of research for generally nonconvex problems.

\section*{Acknowledgments}
The authors would like the acknowledge one of our anonymous reviewers of the SIOPT Journal for the helpful comments on the handling of the infinite-dimensional setting. Moreover, we thank Martin Bauer (Florida State University, U.S.A.) for discussions about differential geometry and, in particular, the shape space $B_e$.

\appendix
\section{Well-posedness and bounds of the PDEs}
\label{appendix-A}
In this section, we prove various properties of the state and adjoint equations from the model problem in \cref{sec:model}.
\paragraph{Proof of \cref{lemma:existenceuniquenessoptimalcontrol-stochastic}}
Let $\tr(\cdot)$ denote the trace operator defined on $\partial D$. Let $a_\xi(y,v):=a_{\xi}^{\textup{in}}(y,v)+ a_{\xi}^{\textup{out}}(y,v)$, where $a_{\xi}^{i}(y,v):= \int_{D_i}  \kappa_i(\xi) \nabla y\cdot \nabla v \, \dx$ for $i \in \{\textup{in}, \textup{out}\}$ and $b_\xi(v) := \int_{\partial D} g(\xi) \tr(v) \, \ds$. Then the weak formulation of the boundary value problem \cref{eq:PDE1}-\cref{eq:jumpconditions} is: find $y = y(\xi) \in \hH$ such that
\begin{equation}
\label{wf}
a_\xi(y,v)=b_\xi(v)\quad \ \forall v\in \hH.
\end{equation}
Coercivity and boundedness of $a_\xi(\cdot,\cdot)$ are clear due to \cref{assumption1}. Therefore, by the Lax--Milgram lemma, there exists a unique solution $y=y(\xi)\in \hH$ to \cref{eq:PDE1}-\cref{eq:jumpconditions}. Let $y_g \in H_{\text{av}}^1(D)$ be such that $\tr( y_g) = g.$ Then, with the solution $y$ to \eqref{wf}, and the continuity of the trace mapping (with constant $C_{\tr}$),
\begin{equation}
\label{eq:inequality-rhs}
b_{\xi}(y) = a_\xi(y,y_g) \leq \kappa^{\max} |y|_{H^1(D)} |y_g|_{H^1(D)} \leq C_{\tr}\kappa^{\max} |y|_{H^1(D)} \lVert g\rVert_{H^{1/2}(D)}.
\end{equation}

The inequality \cref{eq:stateequation-apriori-estimate} follows from the Poincar\'{e} inequality with Poincar\'{e} constant $C_p$ and \eqref{eq:inequality-rhs}, since
\begin{equation}\label{eq:inequality-chain}
 \frac{\kappa_{\min}}{C_p^2+1} \lVert y \rVert_{H^1(D)} ^2 \leq a_\xi(y,y) = b_\xi(y) \leq C_{\tr} \kappa^{\max} \lVert y \rVert_{H^1(D)} \lVert g \rVert_{H^{1/2}(\Go)}.
\end{equation}

\paragraph{Proof of \cref{strong-convergence-state-variables}}
Let $V \in C_0^1(D,\R^2)$ be an arbitrary vector field and set $\kappa^t = \kappa \circ F_t^V$. We define the family of energy functionals over $\xi \in \Xi$ by $E_\xi: [0,\infty) \times H_{\text{av}}^1(D) \rightarrow \R$ such that
\begin{equation}
\label{eq:energy_functional}
E_\xi(t,\varphi)\coloneqq \dfrac12 \int_D \kappa^t(\xi) \, \eta(t)| (DF_t^V)^{-T} \nabla\varphi|^2\, \dx- \int_{\partial D} g(\xi)\varphi\, \ds.
\end{equation}
It is easy to show that for almost every $\xi$, $E_\xi$ is twice continuously differentiable with respect to $\varphi$ and the first and second order derivatives, denoted by $d_\varphi E_\xi$ and $d_\varphi^2E_\xi$, respectively, are given by the following expressions:
\begin{align*}
d_\varphi E_\xi(t,\varphi;\psi) &= \int_D \kappa^t(\xi)A(t)\nabla\varphi\cdot \nabla \psi\, \dx - \int_{\partial D} g(\xi)\psi\, \ds,\\
d_\varphi^2 E_\xi(t,\varphi;\psi,\theta) &= \int_D \kappa^t(\xi)A(t)\nabla \psi \cdot \nabla \theta\, \dx.
\end{align*}
where $A(t)\coloneqq \eta(t)(DF_t^V)^{-1}(DF_t^V)^{-T}$.
Now, we show that $y^t$ is the solution of~\eqref{eq:perturbed_equation}.
Using~\Cref{assumption1} (A1), we can bound the second derivative of the energy functional as follows
\begin{equation*}
d_\varphi^2 E_\xi(t,\varphi;\psi,\psi) \geq \kappa^{\min} \int_D A(t)\nabla \psi \cdot\nabla \psi\, \dx.
\end{equation*}
Thanks to~\cite[p.~526]{Delfour-Zolesio-2001}, we know that there exists $\tau$ small enough such that $A(t)$ is bounded. 
Thus, for all $t\in [0,\tau]$, 
\begin{equation}
\label{eq:coercivity_energy_functional}
d_\varphi^2 E_\xi(t,\varphi;\psi,\psi)  \geq c \lVert \psi\rVert^2_{H^1(D)}.
\end{equation}
With this, we have proven that the energy functional $E_\xi$ is strictly convex in $H_{\text{av}}^1(D)$ with respect to $\varphi$. Moreover, the functional is lower semicontinuous and radially unbounded, which allow us to conclude that the problem
\begin{equation}
\label{eq:min_energy}
\min_{\varphi\in H_{\text{av}}^1(D)} E_\xi(t,\varphi)
\end{equation}
has a unique solution for all $t\in[0,\tau_1]$. Then, is it easy to realize that the solution of problem~\eqref{eq:perturbed_equation} coincides to the solution of the problem~\eqref{eq:min_energy}, which can be characterized as $y^t$ satisfying
\begin{equation*}
d_\varphi E_\xi(t,y_t;\psi) = 0\quad \forall\psi\in \hH.
\end{equation*} 
Regarding the differentiability of the energy functional with respect to $t$, we proceed as follows. First of all, by using~\cite[Lemma~2.2]{Sturm2015:1}, we know that $A(t)$ is continuously differentiable for all $t\in[0,\tau_2]$ with $\tau_2>0$ small enough. Thus,
\begin{equation*}
d_{t,\varphi} E_\xi(t,\varphi;\psi)  = \int_D \kappa^t(\xi)A'(t)\nabla\varphi \cdot \nabla \psi\, \dx  + \int_D (\nabla\kappa(\xi)\cdot V) \, A(t)\nabla \varphi\cdot \nabla \psi \, \dx.
\end{equation*}
Since $\kappa(\xi)$ is piecewise constant, we have $\nabla \kappa = 0$ a.e., implying $d_{t,\varphi} E_\xi(t,\varphi;\psi) = \int_D \kappa^t(\xi)A'(t)\nabla\varphi \cdot \nabla \psi\, \dx$. Now, we notice that for $y^t_r := ry + (1-r)y^t$, it follows that
\begin{align*}
\int_0^1 d_\varphi^2E_\xi(t,y^t_r;y^t-y,y^t-y)\, \dr &=  d_\varphi E_\xi(t,y;y^t-y)-d_\varphi E_\xi(t,y^t;y^t-y)\\
&=d_\varphi E_\xi(t,y;y^t-y)-d_\varphi E_\xi(0,y;y^t-y)\\
&=td_{t,\varphi} E_\xi(tr_t,y;y^t-y),
\end{align*}
where we have use the fact that $y^t$ and $y$ are solutions of the problem~\eqref{eq:min_energy} for $t$ and $0$, respectively.  Furthermore, for the last inequality we have used the mean value theorem, which holds for $r_t\in(0,1)$. Then, on one hand thanks to~\Cref{assumption1} (A1) and the fact that $A'(t)$ is continuously differentiable on $[0,\tau_2]$ and therefore bounded for all $t\in[0,\tau_2]$, we have that 
\begin{equation*}
d_{t,\varphi} E_\xi(r,\varphi;\psi) \leq c \int_D \nabla\varphi \cdot \nabla\psi\,\dx \leq c \lVert \varphi\rVert_{H^1(D)} \lVert\psi\rVert_{H^1(D)}.
\end{equation*}
Using this bound together with~\eqref{eq:coercivity_energy_functional} we get that 
\begin{equation}
\label{eq:bound_yt_y}
c\lVert y^t-y\rVert_{H^1(D)}^2\leq t c \, \lVert y\rVert_{H^1(D)}\lVert y^t-y\rVert_{H^1(D)}
\end{equation}
from which we get the desired inequality for $\tau=\min\{\tau_1,\tau_2\}$. The 
final result is obtained by using~\eqref{eq:stateequation-apriori-estimate}.

\paragraph{Proof of 
\cref{lemma:existenceuniquenessadjoint-stochastic}}
Using analogous arguments as in the proof for \cref{lemma:existenceuniquenessoptimalcontrol-stochastic}, the Lax-Milgram Lemma guarantees the existence of a unique solution $p=p(\xi) \in \hH$ to \eqref{eq:adjoint-weak}.
The inequality \cref{eq:adjointequation-apriori-estimate} comes from \begin{align*}
 \frac{\kappa_{\min}}{C_p^2+1} \lVert p(\xi) \rVert_{H^1(D)}^2 \leq \lVert y(\xi)- \bar{y} \rVert_{L^2(D)} \lVert p(\xi) \rVert_{H^1(D)}.
\end{align*}

\section{Shape differentiability}
\label{appendix-B}
We now prove shape differentiability of $J(u,\xi)$ for a fixed realization (see \Cref{definition_ShapeDer}). Following the averaged adjoint method from~\cite{Sturm2015:1}, let us start by considering the function $\mathcal{L}_\xi:[0,\tau]\times H_{\text{av}}^1(D)\times H_{\text{av}}^1(D) \rightarrow \R$ via
\begin{equation}
\label{eq:LL_definition}
\mathcal{L}_\xi(t,\varphi,\psi) =  \dfrac12 \int_D \eta(t)(\varphi - \bar{y}^t)^2 \, \dx + \int_D \kappa^t(\xi) A(t)\nabla \varphi \cdot \nabla \psi \, \dx  - \int_{\partial D} g(\xi) \psi \, \ds,
\end{equation}
where we use the subscript $\xi$ for the dependence of the function on a fixed but arbitrary realization and $\tau>0$ is a constant that is small enough (to be determined during the proof). Moreover, this function can be also rewritten in terms of the energy functional described in~\eqref{eq:energy_functional} as follows:
\begin{equation}
\label{eq:LL_characterization}
\mathcal{L}_\xi(t,\varphi,\psi) = \dfrac12 \int_D \eta(t)(\varphi - \bar{y}^t)^2 \, \dx + d_\varphi E_\xi(t,\varphi;\psi).
\end{equation}
\paragraph{Proof of \cref{thm:shape_differentiability}}
Since many of these computations are similar to~\cite[Theorem~4.6]{Sturm2015:1}, we will simply sketch the arguments. We set $u^t\coloneqq u \circ F_t^V$, $y^t\coloneqq y \circ F_t^V$, $p^t\coloneqq p \circ F_t^V$, $\bar{y}^t \coloneqq \bar{y}\circ F_t^V$ and $\kappa^t \coloneqq \kappa \circ F_t^V$. In the following, $\xi \in \Xi$ is arbitrary but fixed, and $\tau>0$ is chosen to be small enough. 

Let us start by considering the following: for all $t\in[0,\tau]$ and $\tilde{p}\in H^1_{\text{av}}(D)$, the mapping
\[
[0,1]\rightarrow \R\colon s\mapsto \mathcal{L}_\xi(t,sy^t+(1-s)y^0,\tilde{p})
\]
is absolutely continuous thanks to the characterization~\eqref{eq:LL_characterization} and the fact that in \cref{strong-convergence-state-variables}, we proved the function $E_\xi(t,\varphi)$ is twice continuously differentiable. Additionally, for all $t\in[0,\tau]$, $\varphi \in \hH$, and $\tilde{p}\in \hH$, 
\[
s\mapsto d_{\varphi}\mathcal{L}_\xi(t,sy^t+(1-s)y^0,\tilde{p};\varphi)
\]
is well-defined and belongs to $L^1(0,1)$. With that, Assumption (H0) of~\cite[Sec.~3.1]{Sturm2015:1} is fulfilled.

Additionally, we consider the solution set of the state equation for $t \in [0,\tau]$, given by
\begin{equation*}
\E(t)\coloneqq  \{y\in H_{\text{av}}^1(D) \,|\, d_{\psi}\mathcal{L}_\xi(t,y,0;\hat{\psi})=0 \quad \forall \hat{\psi}\in H_{\text{av}}^1(D)\},
\end{equation*}

For $t\in[0,\tau]$, $y^t\in \E(t)$ and $y^0\in \E(0)$, we define the solution set of the \textit{averaged adjoint equation} with respect to $t,y^t$, and $y^0$ via 
\begin{equation*}
\YY(t,y^t,y^0)\coloneqq \left\{q\in \hH \, \Big|  \int_0^1d_{\varphi}\mathcal{L}_\xi(t,sy^t+(1-s)y^0,q;\hat{\varphi}) \,\ds = 0 \,\,\forall \hat{\varphi} \in \hH \right\}.
\end{equation*}
 Furthermore, for $t=0$ the set $\YY(0,y^0)\coloneqq \YY(0,y^0,y^0)$ coincides with the solution set of the usual adjoint equation, i.e.
\[
\YY(0,y^0) = \left\{q \in \hH \, | \,d_\varphi \mathcal{L}_\xi(0,y^0,q;\hat{\varphi})=0 \quad\forall \hat{\varphi}\in \hH\right\}.
\] 

Now, we will prove the following statements:
\begin{enumerate}[label=\emph{(H\arabic*)}]
\item For all $t\in[0,\tau]$ and all $(y,p)\in \E(0)\times \hH$, the derivative $d_t \mathcal{L}_\xi(t,y,p)$ exists.
\item For all $t\in[0,\tau]$, the set $\YY(t,y^t,y^0)$ is nonempty and $\YY(0,y^0)$ is single-valued.
\item Let $p^0 \in \YY(0,y^0)$. For every sequence $\{t_n\}$ of nonnegative real numbers converging to zero, there exists a subsequence $\{t_{n_k}\}$ such that for all $k$, 
\[
p^{t_{n_k}}\in \YY(t_{n_k},y^{t_{n_k}},y^0)\quad \text{ and }\quad
\lim_{\substack{k\to\infty\\ s\searrow 0}} d_t \mathcal{L}_\xi(s,y^0,p^{t_{n_k}}) = d_t \mathcal{L}_\xi(0,y^0,p^0).
\]
\end{enumerate}

Condition \emph{(H1)} is satisfied as a byproduct of \cref{strong-convergence-state-variables}, since we obtained that the set $\E(t)$ is single-valued for all $t \in [0,\tau]$. Moreover, the function $d_\varphi E_\xi(t,\varphi;\psi)$ is continuously differentiable in $t$ for all $t\in[0,\tau]$. Thanks to~\cite[Lemma~2.1]{Sturm2015:1} we know that $\eta(t)$ is continuously differentiable and therefore we obtain the differentiability of $\mathcal{L}_\xi(t,u,p)$ with respect to $t$ for all $y\in\E(0)$ and $p\in H^1_{\text{av}}(D)$.

Now, we analyze condition \emph{(H2)}. For this, we consider the equation
\begin{align*}
0 	&= \int_0^1d_\varphi \mathcal{L}_\xi(t,ry^t+(1-r)y^0,q;\varphi)\, \dr \\
	&=\int_0^1\int_D \eta(t)(y^0+r(y^t-y^0)-\bar{y}^t)\varphi \, \dx\, \dr+ \int_0^1\int_D\kappa^t(\xi) A(t)\nabla \varphi \cdot\nabla q \, \dx \, \dr.
\end{align*}
By rearranging terms, and integrating with respect to $r$, we obtain the following variational problem: find $q \in H^1_{\text{av}}(D)$ such that
\begin{equation}
\label{eq:averaged_adjoint_weak}
\begin{split}
&\int_D \kappa^t(\xi) A(t)\nabla q \cdot\nabla \varphi\, \dx \\&= -\int_D\eta(t)\left(y^0 + \frac{1}{2}(y^t-y^0)-\bar{y}^t\right)\, \varphi\, \dx \quad\forall \varphi\in H^1_{\text{av}}(D).
\end{split}
\end{equation}
The bilinear form associated with the left-hand side is coercive thanks to~\eqref{eq:coercivity_energy_functional} and the right-hand side makes up a bounded linear form. Then, thanks to the Lax-Milgram lemma, we obtain the existence and uniqueness of solutions, which can be understood as the set $\YY(t,y^t,y^0)$ for all $t\in[0,\tau]$.  In the special case when $t=0$, the solution coincides with the adjoint problem given in~\eqref{eq:adjoint-weak}.

Finally, for the verification of condition \emph{(H3)}, we will prove the following: 
\emph{For every sequence $\{t_n\}$ of nonnegative real numbers converging to zero, there is a subsequence $\{t_{n_k}\}$ such that $\{p^{t_{n_k}}\}$, where $p^{t_{n_k}}$ solves~\eqref{eq:averaged_adjoint_weak} with $t=t_{n_k}$, converges weakly in $\hH$ to the solution $p$ of the adjoint equation~\eqref{eq:adjoint-weak}}.
Therefore, we consider $p^0\in \YY(0,y^0)$ and a nonnegative sequence $\{t_n\}$ converging to zero. Thanks to the verification of condition \emph{(H2)}, we know that there exists a solution for~\eqref{eq:averaged_adjoint_weak}. If we use in particular $\varphi=p^t$ as a test function we get 
\begin{align*}
c \,  \lVert p^t\rVert^2_{H^1(D)}&\leq \int_D \kappa^t A(t)\nabla p^t \cdot \nabla p^t \, \dx= -\int_D\eta(t)(y^0 + \frac{1}{2}(y^t-y^0)-\bar{y}) p^t\, \dx \\
&\leq c \left[\lVert y^0\rVert_{H^1(D)} + \frac12\lVert y^t-y^0\rVert_{H^1(D)}+\lVert \bar{y}\rVert_{H^1(D)}\right]\lVert p^t\rVert_{H^1(D)}\\
&\leq c \left[\lVert y^0\rVert_{H^1(D)} + c\tau \lVert y^0\rVert_{H^1(D)} + \lVert \bar{y}\rVert_{H^1(D)} \right]\lVert p^t\rVert_{H^1(D)},
\end{align*} 
where we have used the Poincar\'{e} inequality, and the boundedness of $\eta(t)$ and $A(t)$ given in~\cite[p.~526]{Delfour-Zolesio-2001} together with~\Cref{assumption1}~(A1) and~\cref{strong-convergence-state-variables}. For the second line, we have used H\"{o}lder's inequality and finally for the third line~\Cref{strong-convergence-state-variables} and the fact that $t\in[0,\tau]$. We conclude that the sequence $\{p^{t_n}\}$ is bounded and therefore we can extract a weakly convergent subsequence, and we denote its weak limit by $w\in\hH$. On the other hand, by \eqref{eq:averaged_adjoint_weak}, for all $k$ we have
\begin{equation*}
\int_D\kappa^{t_{n_k}}(\xi) A(t_{n_k})\nabla p^{t_{n_k}}\cdot \nabla \varphi \, \dx + \int_D\eta(t_{n_k})\left(y^0+\frac12 (y^{t_{n_k}}-y^0)-\bar{y}\right)\varphi \, \dx=0.
\end{equation*} 
Taking the limit as $k\to \infty$, since $y^{t_{n_k}}\to y^0$ in $\hH$, we get
\begin{equation*}
\int_D \kappa(\xi) \nabla w \cdot \nabla \varphi \, \dx + \int_D (y^0-\bar{y})\varphi \, \dx= 0.
\end{equation*}
Since $\YY(0,y^0)$ is single-valued, we conclude that $w = p^0$. Finally, we note that for a fixed $\varphi\in\hH$ the mapping $(t,\psi)\mapsto d_t\mathcal{L}_\xi(t,\varphi,\psi)$ is weakly continuous, from which we conclude that condition \emph{(H3)} is satisfied. 

\bibliographystyle{plain}
\bibliography{references}

\begin{thebibliography}{10}

\bibitem{Absil}
P.A. Absil, R.~Mahony, and R.~Sepulchre.
\newblock {\em {Optimization Algorithms on Matrix Manifolds}}.
\newblock Princeton University Press, 2008.

\bibitem{allaire2015deterministic}
G.~Allaire and C.~Dapogny.
\newblock A deterministic approximation method in shape optimization under
  random uncertainties.
\newblock {\em Journal of computational mathematics}, 1:83--143, 2015.

\bibitem{Alt2012}
Hans~Wilhelm Alt.
\newblock {\em Lineare Funktionalanalysis}.
\newblock Springer, Berlin, Heidelberg, 2012.

\bibitem{Atwal2012}
P.~Atwal, S.~Conti, B.~Geihe, M.~Pach, M.~Rumpf, and R.~Schultz.
\newblock On shape optimization with stochastic loadings.
\newblock In {\em Constrained Optimization and Optimal Control for Partial
  Differential Equations}, volume 160 of {\em Internat. Ser. Numer. Math.},
  pages 215--243. Birkh\"auser/Springer Basel AG, Basel, 2012.

\bibitem{Bauer}
M.~Bauer.
\newblock {\em {Almost Local Metrics on Shape Space}}.
\newblock PhD thesis, Universit\"{a}t Wien, 2010.

\bibitem{bauer2014overview}
M.~Bauer, M.~Bruveris, and P.W. Michor.
\newblock Overview of the geometries of shape spaces and diffeomorphism groups.
\newblock {\em Journal of Mathematical Imaging and Vision}, 50(1-2):60--97,
  2014.

\bibitem{BauerHarmsMichor}
M.~Bauer, P.~Harms, and P.M. Michor.
\newblock Sobolev metrics on shape space of surfaces.
\newblock {\em Journal of Geometric Mechanics}, 3(4):389--438, 2011.

\bibitem{BauerHarmsMichor_SobolevII}
M.~Bauer, P.~Harms, and P.W. Michor.
\newblock {Sobolev metrics on shape space II: Weighted Sobolev metrics and
  almost local metrics}.
\newblock {\em Journal of Geometric Mechanics}, 4(4):365--383, 2012.

\bibitem{Bellido2017}
J.C. Bellido, G.~Buttazzo, and B.~Velichkov.
\newblock Worst-case shape optimization for the {D}irichlet energy.
\newblock {\em Nonlinear Anal. Theory Methods Appl.}, 153:117--129, 2017.

\bibitem{Bonnabel2011}
S.~Bonnabel.
\newblock Stochastic gradient descent on {R}iemannian manifolds.
\newblock {\em IEEE Trans. Automat. Contr.}, 58(9):2217--2229, 2013.

\bibitem{Bottou1998}
L{\'e}on Bottou.
\newblock Online learning and stochastic approximations.
\newblock {\em On-Line Learning in Neural Networks}, 17(9):142, 1998.

\bibitem{Bruegger2018}
R.~Br\"ugger, R.~Croce, and H.~Harbrecht.
\newblock Solving a {B}ernoulli type free boundary problem with random
  diffusion.
\newblock {\em Preprint No. 2018-09}, 2018.

\bibitem{Buttazzo2003}
G.~Buttazzo and L.~De Pascale.
\newblock Optimal shapes and masses, and optimal transportation problems.
\newblock {\em Lecture Notes in Math. Springer}, 2003.

\bibitem{Cheney1999}
M.~Cheney, D.~Isaacson, and J.~Newell.
\newblock Electrical impedance tomography.
\newblock {\em SIAM Rev.}, 41(1):85--101, 1999.

\bibitem{constantin2007geodesic}
A.~Constantin, T.~Kappeler, B.~Kolev, and P.~Topalov.
\newblock On geodesic exponential maps of the virasoro group.
\newblock {\em Annals of Global Analysis and Geometry}, 31(2):155--180, 2007.

\bibitem{Conti2008}
S.~Conti, H.~Held, M.~Pach, M.~Rumpf, and R.~Schultz.
\newblock Shape optimization under uncertainty---a stochastic programming
  perspective.
\newblock {\em SIAM J. Optim.}, 19(4):1610--1632, 2008.

\bibitem{Conti2018}
S.~Conti, M.~Rumpf, R.~Schultz, and S.~T\"{o}lkes.
\newblock Stochastic dominance constraints in elastic shape optimization.
\newblock {\em SIAM J. Control Optim.}, 2018.

\bibitem{Costabel2010}
Martin Costabel, Monique Dauge, and Serge Nicaise.
\newblock Corner singularities and analytic regularity for linear elliptic
  systems. part i: Smooth domains.
\newblock 2010.

\bibitem{Dambrine2015}
M.~Dambrine, C.~Dapogny, and H.~Harbrecht.
\newblock Shape optimization for quadratic functionals and states with random
  right-hand sides.
\newblock {\em SIAM J. Control Optim.}, 2015.

\bibitem{dambrine2019incorporating}
M.~Dambrine, H.~Harbrecht, and B.~Puig.
\newblock Incorporating knowledge on the measurement noise in electrical
  impedance tomography.
\newblock {\em ESAIM: Control, Optimisation and Calculus of Variations}, 25:84,
  2019.

\bibitem{Dambrine2016}
M.~Dambrine and A.~Laurain.
\newblock A first order approach for worst-case shape optimization of the
  compliance for a mixture in the low contrast regime.
\newblock {\em Struct. Multidiscip. Optim.}, 54(2):215--231, 2016.

\bibitem{Davis2018}
Damek Davis, Dmitriy Drusvyatskiy, Sham Kakade, and Jason~D Lee.
\newblock Stochastic subgradient method converges on tame functions.
\newblock {\em Found. Comput. Math.}, pages 1--36, 2018.

\bibitem{Delfour-Zolesio-2001}
M.C. Delfour and J.-P. Zol\'esio.
\newblock {\em {Shapes and Geometries: Metrics, Analysis, Differential
  Calculus, and Optimization}}, volume~22 of {\em Adv. Des. Control}.
\newblock SIAM, 2nd edition, 2001.

\bibitem{Eigel2019}
Martin Eigel, Manuel Marschall, and Michael Multerer.
\newblock An adaptive stochastic {G}alerkin tensor train discretization for
  randomly perturbed domains.
\newblock {\em arXiv preprint arXiv:1902.07753}, 2019.

\bibitem{escherright}
J~Escher and B~Kolev.
\newblock Right-invariant sobolev metrics hs on the diffeomorphisms group of
  the circle.
\newblock {\em Journal of Geometric Mechanics}.

\bibitem{FerSva}
O.P. Ferreira and B.F. Svaiter.
\newblock Kantorovich’s theorem on {N}ewton’s method in {R}iemannian
  manifolds.
\newblock {\em J. Complexity}, 18(1):304–329, 2002.

\bibitem{Langer-2015}
P.~Gangl, A.~Laurain, H.~Meftahi, and K.~Sturm.
\newblock Shape optimization of an electric motor subject to nonlinear
  magnetostatics.
\newblock {\em SIAM J. Sci. Comput.}, 37(6):B1002--B1025, 2015.

\bibitem{Geiersbach2020a}
Caroline Geiersbach.
\newblock {\em Stochastic Approximation for PDE-Constrained Optimization under
  Uncertainty}.
\newblock PhD thesis, University of Vienna, 2020.

\bibitem{Geiersbach2019}
Caroline Geiersbach and Georg~Ch Pflug.
\newblock Projected stochastic gradients for convex constrained problems in
  {H}ilbert spaces.
\newblock {\em SIAM J. Optim.}, 29(3):2079--2099, 2019.

\bibitem{Geiersbach2020}
Caroline Geiersbach and Teresa Scarinci.
\newblock Stochastic proximal gradient methods for nonconvex problems in
  {H}ilbert spaces.
\newblock {\em https://arxiv.org/abs/2001.01329}, 2020.

\bibitem{Gut2013}
Allan Gut.
\newblock {\em Probability: a graduate course}, volume~75.
\newblock Springer Science \& Business Media, 2013.

\bibitem{Haber2012}
E.~Haber, M.~Chung, and F.~Herrmann.
\newblock An effective method for parameter estimation with {PDE} constraints
  with multiple right-hand sides.
\newblock {\em {SIAM} J. Optim.}, 22(3), 2012.

\bibitem{Harbrecht2018}
H.~Harbrecht and M.D. Peters.
\newblock The second order perturbation approach for elliptic partial
  differential equations on random domains.
\newblock {\em Appl. Numer. Math.}, 125:159--171, 2018.

\bibitem{Ito-Kunisch-Peichl}
K.~Ito, K.~Kunisch, and G.H. Peichl.
\newblock Variational approach to shape derivatives.
\newblock {\em ESAIM Control Optim. Calc. Var.}, 14(3):517--539, 2008.

\bibitem{Keshavarzzadeh2017}
Vahid Keshavarzzadeh, Felipe Fernandez, and Daniel~A Tortorelli.
\newblock Topology optimization under uncertainty via non-intrusive polynomial
  chaos expansion.
\newblock {\em Comp. Methods Appl. Mech. Eng.}, 318:120--147, 2017.

\bibitem{KrieglMichor}
A.~Kriegl and P.~Michor.
\newblock {\em {The Convient Setting of Global Analysis}}, volume~53 of {\em
  Mathematical Surveys and Monographs}.
\newblock American Mathematical Society, 1997.

\bibitem{Kwon2002}
O.~Kwon, E.~Je Woo, J.R. Yoon, and J.K. Seo.
\newblock Magnetic resonance electrical impedance tomography ({MREIT}):
  Simulation study of ${J}$-substitution algorithm.
\newblock {\em IEEE Trans. Biomed. Eng.}, 49(2), 2002.

\bibitem{Lang}
S.~Lang.
\newblock {\em {Fundamentals in Differential Geometry}}, volume 191 of {\em
  Grad. Texts in Math.}
\newblock Springer, 2nd edition, 2001.

\bibitem{Lee}
J.M. Lee.
\newblock {\em {Manifolds and Differential Geometry}}, volume 107 of {\em Grad.
  Stud. Math.}
\newblock Amer. Math. Soc., 2009.

\bibitem{Lee2018}
John Lee.
\newblock {\em Introduction to Riemannian Manifolds}.
\newblock Springer International Publishing, 2nd edition, 2018.

\bibitem{Liu2017}
Dishi Liu, Alexander Litvinenko, Claudia Schillings, and Volker Schulz.
\newblock Quantification of airfoil geometry-induced aerodynamic
  uncertainties---comparison of approaches.
\newblock {\em SIAM-ASA J. Uncertain.}, 5(1):334--352, 2017.

\bibitem{Lord2014}
G.~Lord, C.~Powell, and T.~Shardlow.
\newblock {\em An Introduction to Computational Stochastic {PDE}s}.
\newblock Cambridge University Press, 2014.

\bibitem{LuftWelker}
Daniel Luft and Kathrin Welker.
\newblock Computational investigations of an obstacle-type shape optimization
  problem in the space of smooth shapes.
\newblock In {\em International Conference on Geometric Science of
  Information}, pages 579--588. Springer, 2019.

\bibitem{Martin2018}
M.~Martin, S.~Krumscheid, and F.~Nobile.
\newblock Analysis of stochastic gradient methods for {PDE}-constrained optimal
  control problems with uncertain parameters.
\newblock Technical report, \'Ecole Polytechnique {MATHICSE} Institute of
  Mathematics, 2018.

\bibitem{Martinez-Frutos2016}
J.~Mart{\'i}nez-Frutos, D.~Herrero-P{\'e}rez, M.~Kessler, and F.~Periago.
\newblock Robust shape optimization of continuous structures via the level set
  method.
\newblock {\em Comput. Methods Appl. Mech. Engrg.}, 305:271--291, 2016.

\bibitem{Martinez-Frutos2018}
Jes{\'u}s Mart{\'i}nez-Frutos and Francisco~Periago Esparza.
\newblock {\em Optimal Control of {PDE}s Under Uncertainty: An Introduction
  with Application to Optimal Shape Design of Structures}.
\newblock Springer, 2018.

\bibitem{Metivier2011}
Michel M{\'e}tivier.
\newblock {\em Semimartingales: a Course on Stochastic Processes}, volume~2.
\newblock Walter de Gruyter, 2011.

\bibitem{MichorMumford2}
P.M. Michor and D.~Mumford.
\newblock Vanishing geodesic distance on spaces of submanifolds and
  diffeomorphisms.
\newblock {\em Doc. Math.}, 10:217--245, 2005.

\bibitem{MichorMumford1}
P.M. Michor and D.~Mumford.
\newblock Riemannian geometries on spaces of plane curves.
\newblock {\em J. Eur. Math. Soc. (JEMS)}, 8(1):1--48, 2006.

\bibitem{MichorMumford}
P.M. Michor and D.~Mumford.
\newblock {An overview of the Riemannian metrics on spaces of curves using the
  Hamiltonian approach}.
\newblock {\em Appl. Comput. Harmon. Anal.}, 23(1):74--113, 2007.

\bibitem{Pflug1996}
Georg~Ch. Pflug.
\newblock {\em Optimization of Stochastic Models: The Interface Between
  Simulation and Optimization}.
\newblock Springer, 1996.

\bibitem{Pironneau1984}
O.~Pironneau.
\newblock {\em Optimal shape design for elliptic systems}.
\newblock Springer-Verlag, 1984.

\bibitem{Ring2012}
Wolfgang Ring and Benedikt Wirth.
\newblock Optimization methods on {R}iemannian manifolds and their application
  to shape space.
\newblock {\em SIAM J. Optim.}, 22(2):596--627, 2012.

\bibitem{Robbins1951}
H.~Robbins and S.~Monro.
\newblock A stochastic approximation method.
\newblock {\em Ann. Math. Statist.}, 22(3):400--407, 1951.

\bibitem{Rockafellar2015}
R.T. Rockafellar and J.O. Royset.
\newblock Engineering decisions under risk averseness.
\newblock {\em ASCE ASME J. Risk Uncertain. Eng. Syst. A Civ. Eng.}, 1(2),
  2015.

\bibitem{SchulzSiebenborn2016}
V.~Schulz and M.~Siebenborn.
\newblock Computational comparison of surface metrics for {PDE} constrained
  shape optimization.
\newblock {\em Comput. {M}ethods {A}ppl. {M}ath.}, 16(3):485--496, 2016.

\bibitem{SchulzWelker}
V.~Schulz and K.~Welker.
\newblock On optimization transfer operators in shape spaces.
\newblock In {\em Shape Optimization, Homogenization and Optimal Control},
  pages 259--275. Springer, 2018.

\bibitem{Schulz}
V.H. Schulz.
\newblock {A Riemannian view on shape optimization}.
\newblock {\em Found. Comput. Math.}, 14(3):483--501, 2014.

\bibitem{Schulz2015a}
V.H. Schulz, M.~Siebenborn, and K.~Welker.
\newblock Structured inverse modeling in parabolic diffusion problems.
\newblock {\em SIAM J. Control Optim.}, 53(6):3319--3338, 2015.

\bibitem{SchulzSiebenbornWelker2015:2}
V.H. Schulz, M.~Siebenborn, and K.~Welker.
\newblock Efficient {PDE} constrained shape optimization based on
  {S}teklov-{P}oincar{\'e} type metrics.
\newblock {\em SIAM J. Optim.}, 26(4):2800--2819, 2016.

\bibitem{Siebenborn2017}
M.~Siebenborn and K.~Welker.
\newblock Algorithmic aspects of multigrid methods for optimization in shape
  spaces.
\newblock {\em SIAM J. Sci. Comput.}, 39(6):B1156--B1177, 2017.

\bibitem{Sokolowski1991}
J.~Sokolowski and J.~Zol\'esio.
\newblock {\em Introduction to Shape Optimization: Shape Sensitivity Analysis}.
\newblock Springer-Verlag, 1991.

\bibitem{srivastava2016functional}
A.~Srivastava and E.P. Klassen.
\newblock {\em Functional and shape data analysis}, volume~1.
\newblock Springer, 2016.

\bibitem{Sturm2015:1}
K.~Sturm.
\newblock Minimax {L}agrangian approach to the differentiability of nonlinear
  {PDE} constrained shape functions without saddle point assumptions.
\newblock {\em SIAM J. Control Optim.}, 53(4):2017--2039, 2015.

\bibitem{Sturm2014}
Kevin Sturm.
\newblock {\em On Shape Optimization with Non-Linear Partial Differential
  Equations}.
\newblock PhD thesis, Technische Universit{\"a}t Berlin, 2014.

\bibitem{Welker2016}
K.~Welker.
\newblock {\em Efficient {PDE} Constrained Shape Optimization in Shape Spaces}.
\newblock PhD thesis, Universit\"{a}t Trier, 2016.

\bibitem{younes2010shapes}
L.~Younes.
\newblock {\em Shapes and diffeomorphisms}, volume 171.
\newblock Springer, 2010.

\bibitem{ZhangSra2016}
H.~Zhang and S.~Sra.
\newblock First-order methods for geodesically convex optimization.
\newblock In {\em Conference on Learning Theory}, pages 1617--1638, 2016.

\end{thebibliography}
\end{document}